\documentclass[reqno,12pt,letterpaper]{amsart}
\usepackage[proof]{newmzmacros1}
\usepackage[all,cmtip]{xy}
\newcommand{\e}{\epsilon}
\newcommand{\RR}{{\mathbb R}}

\newcommand{\CC}{{\mathbb C}}

\newcommand{\CI}{{C^\infty}}

\newcommand{\TT}{{\mathbb T}}
\usepackage[normalem]{ulem}

\title{An introduction to microlocal complex deformations}

\author{Jeffrey Galkowski}
\email{j.galkowski@ucl.ac.uk}
\address{Department of Mathematics, University College London, WC1H 0AY, UK}

\author{Maciej Zworski} 
\email{zworski@math.berkeley.edu}
\address{Department of Mathematics, University of California, Berkeley, CA 94720}

\begin{document}


\begin{abstract}
In this expository article we relate the presentation of weighted estimates in \cite{M} to the Bergman kernel approach of \cite{Sj96}. It is meant as an introduction to the Helffer--Sj\"ostrand theory
\cite{HS} in the simplest setting and to its adaptations to compact 
manifolds \cite{Sj96}, \cite{GZ}.
\end{abstract}

\maketitle

\section{Introduction}
\label{s:intr}

Suppose that $ P $ is a semiclassical differential operator (or  a pseudodifferential operator, see 
\eqref{eq:pseudo1}), for instance, 
\begin{equation}
\label{eq:Schr} P = - h^2 \Delta + V . 
\end{equation}
Conjugation by exponential weights has a very long tradition going back to the
origins of Carleman and Agmon--Lithner estimates:
\begin{equation}
\label{eq:Pphi0}  P_\varphi := e^{ \varphi (x)/h} P e^{ - \varphi ( x)/h}, \ \ \varphi \in C^\infty ( \RR^n ; \RR ) , 
\end{equation}
which in the case of \eqref{eq:Schr} gives (with $ D_x = -i \partial_x $, 
and $ v^2 := v_1^2 + \cdots + v_n^2 $, $ v\in \CC^n $)
\[\begin{split}  P_\varphi  & = ( h D_x + i \nabla \varphi )^2 + V ( x )
\\
&  = -h^2 \Delta +  \nabla \varphi \cdot h \nabla 
 +  V - |\nabla \varphi |^2  + h \Delta \varphi  .
 \end{split}  \]
Roughly speaking, exploiting the sign of $ V - | \nabla \varphi |^2 $
leads to exponential decay (tunneling) estimates for solutions of
$ P u = 0 $ -- see for instance \cite[\S 7.1]{zw} and references given there.
For exponential lower bounds (quanitative unique continuation, from the mathematical point of view), one exploits positivity properties
of $ [ P_\varphi^* , P_\varphi ] $
of suitably chosen $ \varphi$ using the identity (with $ L^2 $ norms and
$ u \in C^\infty_{\rm{c}} $)
\begin{equation}
\label{eq:PphiS}    \| P_\varphi u \|^2 = \| P_\varphi^* u \|^2 + \langle [ P_\varphi^* , P_\varphi ] u, u \rangle 
 \geq \langle [ P_\varphi^* , P_\varphi ] u, u \rangle, 
\end{equation}
see for instance \cite[\S 7.2]{zw}. 

On the other hand conjugation \eqref{eq:Pphi0}
with $ \varphi ( x ) $ replaced by $ i \varphi ( x ) $ gives the simplest 
case of Egorov's Theorem (see for instance \cite[Theorem 11.1]{zw}):
\[ P_{ i \varphi } = ( h D_x - \nabla \varphi ( x) )^2 + V ( x ) , \]
which corresponds to the pull back of the symbol by the canonical transformation
$ ( x, \xi) \mapsto ( x, \xi - \nabla \varphi ( x ) ) $ associated
to the operator $ u ( x ) \mapsto e^{ - \frac i h \varphi( x) } u ( x) $.

{When $\varphi$ is real, we have implicitly used analyticity of $\xi\to\xi^2$ to obtain}  \eqref{eq:PphiS}. If we formally conjugate \eqref{eq:Schr} with $ V ( x ) = x^2 $ by 
$ e^{ \varphi ( h D ) / h } $ we obtain
\[   e^{ \varphi ( h D ) / h } P e^{ - \varphi ( h D ) / h }
= - h^2 \Delta + V ( x - i \nabla \varphi ( h D ) ) , \]
where we used the analyticity of $ V ( x ) = x^2 $. In general, we  encounter problems akin to flowing the heat equation backwards which again requires analyticy.

In many problems it is advantageous to use $ \varphi = G ( x, 
h D_x ) $ but, as the discussion above shows, {the use of such weights}  requires analyticity 
assumptions (unless we use weights of moderate growth in $ h $ and
$ \xi $ -- see \cite[\S 8.2]{zw} for a textbook discussion 
 and Faure--Sj\"ostrand \cite{FaS}, Dyatlov--Zworski \cite{DZ} for recent applications and references). 
 
 The use of strong microlocal weights ($ \varphi$ in some sense equal to 
 $ G ( x, h D ) $) has been raised to the level of high art by 
 Sj\"ostrand and his collaborators -- see for instance Hitrik--Sj\"ostrand \cite{HiS} and 
 references given there. Here we would like to concentrate on the approach 
 motivated by scattering resonances and introduced by Helffer--Sj\"ostrand \cite{HS}.
 
\renewcommand\thefootnote{\dag}%
 
The goal then is to justify the statement
\begin{equation}
\label{eq:GPG}  
\begin{split} 
 e^{ - G( x , h D ) / h } P ( x, h D ) e^{  G ( x, h D ) / h} 
& \sim P ( x + i \nabla_\xi G ( x, h D ) , h D_x - i \nabla_x G ( x, h D ) ) \\
& = P ( x , h D_x ) - i H_P G ( x , h D ) + \mathcal O ( \| G \|^2_{C^2 } ) ,\end{split} \end{equation}
and then to exploit the possible gain of ellipticity for the right hand side. In particular, the property $ H_P G ( x, \xi ) > 0 $ can be used
to great advantage. Here 
\[  H_P  := \sum_{j=1}^n \partial_{\xi_j} P ( x, \xi ) \partial_{x_j }
- \partial_{ x_j } P ( x, \xi ) \partial_{\xi_j } ,\]
is the Hamilton vector field of the symbol of $ P $ and 
$  \mathcal O ( \| G \|^2_{C^2 } )  $ means a norm bound between suitable
spaces, for instance $ L^2 \to L^2 $ if $ P $ is order $ 0 $ and $ G $ of
order $ 1 $.
The condition $ H_P G > 0 $ and its weaker forms
are called the {\em escape function} property or the {\em positive commutator} property.
 
One tool for justifying \eqref{eq:GPG} is the FBI transform\footnote{It is named after Fourier--Bros--Iagolnitzer and this name 
is used for its
generalizations in microlocal analysis. In specific cases, and in other
fields it is called Bargmann, Segal, Gabor, and wave packet transform.} -- 
see \eqref{eq:defbi} and \cite{HiS}, \cite{M}, \cite[Chapter 13]{zw} for three introductions. Roughly speaking it turns the action of the operator $ P$  to multiplication by its symbol (say, $ \xi^2 + V ( x ) $, 
in the case of \eqref{eq:Schr}). When weights are introduced, 
this action turns into multiplication by the ``deformed symbol".
That is, roughly speaking, the symbol of the operator on 
right hand side of \eqref{eq:GPG}.

\renewcommand\thefootnote{\ddag}%

Here we will present the simplest case of small (in $ C^2 $) compactly 
supported weights $ G $. A very clear presentation (without the smallness assumption) following
Nakamura \cite{N} is provided in Martinez \cite[\S 3.5]{M} but our goal is to make {\em simple things complicated} by explaining the theory in the 
way which adapts to the case of stronger (non-compactly supported) weights
used in \cite{HS} and to the case of compactly supported weights on 
compact manifolds of \cite{Sj96}. Our motivation comes from the study of 
viscosity limits for 0th order (analytic) pseudodifferential operators \cite{GZ}. It partly justifies claims made in the physics literature, see for instance 
\cite{Rieu}\footnote{Specifically, the claim that 
``The aim of this paper is to present what we believe to be the asymptotic limit of inertial modes in a spherical shell when viscosity tends to zero." These viscosity limits are essentially the {\em resonances} of zero order operators and hence it is natural to use the methods of
\cite{HS}. Except in simplest cases the methods based on spacial deformations in the spirit of complex scaling -- see \cite[\S 4.5, \S 4.7]{dizzy} and references given there -- are not sufficient.}.

A properly interpreted version of \eqref{eq:GPG} is given 
in Theorem \ref{t:2} which comes in this form from \cite{N}, \cite{M}.
The proof however follows the strategy of \cite{Sj96} and is based
on the study of orthogonal projections onto weighted spaces of (essentially) holomorphic functions. Theorem \ref{t:3} presents a more geometric version more directly in the spirit of \cite{Sj96}.

Our exposition of this material is structured as follows:

\begin{itemize}

\item In \S \ref{s:FBI} we review the properties of the FBI (Bargmann/Segal/Gabor/wave packet) transform and the structure of pseudodifferential operator on the FBI transform side. No (non-quadratic) weights enter here but the simple geometric structure discussed in \S \ref{s:geop}
provides a guide for more complicated constructions.

\item \S \ref{s:we} is devoted to the description of the projector onto the image of the FBI transform, orthogonal with respect to the norm on 
$ L^2 ( T^* \RR^n , e^{-\varphi(x,\xi)/h} dx d\xi ) $. That follows the approach of \cite{Sj96} which in turn is inspired by \cite{BS}, \cite{BG} and \cite{HS}. The description of the action of {\em analytic}
pseudodifferential operators on those spaces is then given in 
Theorem \ref{t:2} in \S \ref{s:psew}. 

\item \S \ref{s:anex} reviews some aspects of the  analytic machinery of Melin--Sj\"ostrand \cite{mess} which is needed for the more geometric approach to the justification of \eqref{eq:GPG} in 
\S\S  \ref{s:projdefo},\ref{s:pseudoL}. 

\item A more geometric version, following the spirit of \cite{HS} and
\cite{Sj96}, is presented in \S \ref{s:projdefo}. Instead of putting in  a weight, the phase space $ T^* \RR^n $ is deformed to 
$ \Lambda := \{ ( x + i G_x ( x, \xi ) , \xi - i G_\xi ( x, \xi ) ):
( x, \xi ) \in T^* \RR^n \} $ (note the analogy with the right hand side
\eqref{eq:GPG}; $ \Lambda $ is always Lagrangian with respect to 
$ \Im d\zeta \wedge dz $ on $ T^* \CC^n $ and symplectic for 
$ \Re d \zeta \wedge dz $ for $ G $ sufficiently small). That corresponds to continuing the FBI transform analytically and then restricting it to $ \Lambda $. The action of an
{\em analytic} pseudodifferential operator with symbol $ p$ on that space is (in some sense) close to multiplication by $ p|_\Lambda $ -- see
Theorem \ref{t:3}. That is again achieved by constructing an appropriate orthogonal projector.

\item Finally, in \S \ref{s:wvsd} we discuss the equivalence of the two approaches by showing that each deformation $ \Lambda $ corresponds
to putting in a weight without a deformation -- see Theorem \ref{th:G2ph}.

\end{itemize}

\medskip\noindent\textbf{Acknowledgements.}
We would like to thank Semyon Dyatlov for many enlightening discussions and 
Johannes Sj\"ostrand for helpful comments on the first version of this article. 
Partial support by the National Science Foundation grants
DMS-1900434 and DMS-1502661 (JG) and 
 DMS-1500852 (MZ) is also
gratefully acknowledged.

\section{The FBI transform}
\label{s:FBI}

We define the usual FBI transform:
\begin{equation}
\label{eq:defbi} 
\begin{gathered}
T u ( x , \xi ) := c h^{-\frac {3n}4} \int_{\RR^n} e^{ \frac i h  (
\langle x - y , \xi \rangle + \tfrac i 2  ( x - y )^2 ) } 
u ( y ) d y .
\end{gathered}
\end{equation}

Then $ T : L^2 ( \RR^n ) \to L^2 ( T^* \RR^n ) $ is an isometry 
as is easily checked using Plancherel's formula -- see for instance \cite[Step 2 of the proof Theorem 13.7]{zw}. 
We then notice that  
\begin{equation}
\label{eq:imag}
\begin{gathered}
T ( L^2 ( \RR^n ) ) \subset \mathscr H := \{ u \in L^2 ( T^* \RR^n ) : 
Z_j u = 0 , \ j = 1, \cdots n  \}, \\
\zeta_j ( x, \xi, x^* , \xi^*) := x^*_j - \xi_j - i \xi^*_j , \\
Z_j := \zeta_j ( x,\xi, hD_x, hD_\xi ) = e^{- \xi^2/2h}{2h} {D}_{{\bar{z}}_j} e^{ \xi^2/2h } , \\
z_j = x_j - i \xi_j,\qquad {D_{\bar{z}_j}=\tfrac{1}{2i}(\partial_{x_j}-i\partial_{\xi_j}).}
\end{gathered}
\end{equation}
In fact, the range of $ T $ is exactly given by $ \mathscr H $:

\begin{prop}
\label{p:TTs}
The orthogonal projector $  \Pi_0 : L^2 ( T^* \RR^n) \to \mathscr H $ is given by 
$ \Pi_0 = T T^* $ and 
\begin{equation}
\label{eq:Pi0} \begin{gathered}
 \Pi_0 u ( \alpha ) =  h^{-n} c_0 \int_{T^* \RR^n }  e^{ \frac i h \psi_0 ( \alpha, \beta ) }  u ( \beta ) d \beta , \ \
 \alpha = ( x, \xi) ,  \ \beta = ( x', \xi' ) , \\
 \psi_0 ( \alpha, \beta ) = \tfrac12 (x\xi - x'\xi') + \tfrac12 ( x 
\xi' - \xi x') + \tfrac i 4 (x-x')^2 + \tfrac i 4 ( \xi - \xi' )^2 . 
\end{gathered}
\end{equation}
In particular, $ T ( L^2 ( \RR^n ) ) = \mathscr H $. 
\end{prop}
For the proof see \cite[Theorem 13.7]{zw} or \cite[Exercise 3.6.2]{M}.

\noindent{{\bf{Remark:}} Note that, using the holomorphic notation $z=x-i\xi$, $w=x'-i\xi'$, 
\begin{equation}
\label{eq:holPhase}
\psi_0= i \Big[\Phi_0(z)+\tfrac{1}{2}(z-\bar{w})^2+\Phi_0(w)\Big],\qquad \Phi_0(z):= \tfrac 12 |\Im z|^2.
\end{equation}}

\subsection{Pseudodifferential operators on the FBI side}
\label{s:psefbi}
Suppose $ P  = p( x, h D ) $,
\begin{equation}
\label{eq:pseudo1} \begin{gathered} Pu = p ( x, h D ) u := \frac 1 { ( 2\pi h )^n } 
\int_{\RR^n} \! \int_{\RR^n} p \left( x , \xi \right)
e^{ \frac i h \langle x - y , \xi \rangle } u ( y) dy d \xi , \\
| \partial^\alpha_{x,\xi} p ( x, \xi ) | \leq C_\alpha . 
\end{gathered}
\end{equation}
 is a pseudodifferential operator (with the symbol, $p \in  S (1 ) $ in the 
 terminology of \cite[Chapter 4]{zw}). 
We want to consider 
\begin{equation}
\label{eq:mathP}  \mathscr P := T P T^* : \mathscr H \to \mathscr H . \end{equation}
There are many ways to think about this operator -- see \cite[\S 13.4]{zw} for Sj\"ostrand's pseudodifferential approach. Here we look at it in the spirit of \cite{Sj96}. 

\begin{lemm}
\label{l:TPT}
The operator \eqref{eq:mathP} is given by
\[ \mathscr P u = \int_{T^* \RR^n } K_P ( \alpha, \beta ) u ( \beta ) d \beta \]
with
\begin{equation}
\label{eq:KPexp}
\begin{gathered} K_P ( \alpha, \beta ) = c_0 h^{-n} e^{ \frac i h \psi_0 ( \alpha, \beta ) } 
a( \alpha, \beta ) + \mathcal O ( h^\infty \langle 
\alpha - \beta \rangle^{-\infty} ) , \\
a ( \alpha, \beta ) \sim \sum_{ j=0}^\infty h^j a_j ( \alpha, \beta ) , \ \  | \partial_\alpha^\gamma \partial_\beta^{\gamma'} a_j ( \alpha, \beta) | \leq C_{\gamma \gamma' j} , \ \ 
a_0 ( \alpha, \alpha) = p ( \alpha ) .
\end{gathered}
\end{equation}
\end{lemm}
\begin{proof}
We calculate the integral kernel using, again, the completion of squares
and integration in $ y'$:
\[ \begin{split} K_P & = \frac{e^{ \frac i h (  x \xi -  x' \xi')  } }{ ( 2 \pi h )^n } 
\int_{ \RR^{3n} } e^{ \frac i h ( y ( \eta - \xi ) - y' ( \eta - \xi') + 
\frac i 2 ( x - y)^2 + \frac i 2 ( x' - y' )^2 ) } 
p \left ( { y }  , \eta \right) dy' dy d\eta \\
& = \frac{e^{ \frac i h   x \xi   } }{ ( 2 \pi h )^{\frac n 2 } }
\int_{ \RR^{3n} } e^{ \frac i h ( y ( \eta - \xi ) - x' \eta  + 
\frac i 2 ( x - y)^2 + \frac i 2 ( \xi' - \eta )^2   ) } 
p \left ( { y }  , \eta \right) dy  d\eta .
\end{split} \]
We note that the integral is now absolutely convergent and, 
denoting the phase by $ \Phi$, $ \Im \Phi \geq 0 $.
The stationary points of  $ \Phi $ are given by solving (or completing squares)
\[  \partial_\eta \Phi =  y - x' + i ( \eta - \xi' ) = 0 , \ \ 
\partial_y \Phi =  \eta - \xi  + i ( y - x ) = 0  , \]
\[  \Phi'' = \begin{bmatrix}   i I_{\RR^n} & I_{\RR^n} \\
 I_{\RR^n}  &  i I_{\RR^n}   \end{bmatrix}  \ \text{ is non-degenerate.} \]
The solutions are 
\[  y = y_c := \tfrac 12 ( x + x' ) + \tfrac i2 ( \xi'-\xi ) , \ \ \eta =
\eta_c :=  \tfrac 12 (  \xi + \xi' ) + \tfrac i2 ( x - x' ) .\]
We note that 
\[ {|\partial_\eta\Phi|^2+|\partial_y\Phi|^2} \geq \tfrac12 | x - x' |^2 + 
\tfrac 12 |\xi - \xi'|^2  . \]
Hence non-stationary phase estimate shows that if we restrict the
integration to $ | y - y_c |^2 + | x- x_c |^2 < 1 $ (using a smooth 
cut-off function), the remaining term is estimated by 
$ \mathcal O ( h^\infty \langle \alpha - \beta \rangle )$, 
$ \alpha = ( x,\xi) $, $ \beta = ( x',\xi') $. 

For the integral over the set
close to the critical points we apply 
the complex stationary phase method \cite[Theorem 2.3, p.148]{mess} to obtain \eqref{eq:KPexp} 
\begin{equation}
\label{eq:apw} 
\begin{split} a & = a_0 + ha_1 + \cdots \\
& = \widetilde p \left( \tfrac 12 ( x + x' ) + \tfrac i 2 ( \xi' - \xi ) , 
\tfrac 12 (  \xi + \xi' ) + \tfrac i 2 ( x - x' ) \right) + \mathcal O ( h ) ,
\end{split}  \end{equation}
where $ \widetilde p $ is an almost analytic extension of $ p $. 
We note that $  a_0 ( \alpha, \alpha ) = p ( \alpha )$.
\end{proof}

Also, just as for the kernel of $ \Pi_0 $, 
$ K ( \alpha , \beta ) $ has to satisfy the equations
\begin{equation}
\label{eq:Kab0}  \zeta_j ( \alpha, h D_\alpha ) K ( \alpha , \beta ) = 0 , \ \ 
\widetilde \zeta_j ( \beta, h D_\beta ) K ( \alpha, \beta ) = 0  , \ \
\widetilde \zeta_j ( \beta, \beta^* ) := \overline{{\zeta_j} ( \beta, - \beta^* )}  .\end{equation}
The last condition follows from the fact that 
\[ \begin{split} 0 & = ( \zeta_j T P^* T^* )^* u ( \alpha ) = ( \zeta_j \mathscr P^*  )^* u ( \alpha )
= \mathscr P \zeta_j^* u ( \alpha) 
  =  \int K( \alpha, \beta ) \bar \zeta_j  ( \beta, D_\beta) u ( \beta ) d\beta \\
& = \int  \left[ \bar \zeta_j ( \beta, - D_\beta) K ( \alpha, \beta ) \right] u ( \beta ) d\beta . \end{split} \]
That means that
\begin{gather}
\label{eq:eiktr0}  
\begin{gathered} \zeta_j ( \alpha, d_\alpha \psi_0 ( \alpha, \beta) ) = 0 , \ \
\widetilde \zeta_j ( \beta, d_\beta \psi_0 ( \alpha, \beta ) ) = 0 , \ \ 
\bar \partial_{z_j} a_k , \ \partial_{w_j}  a_k = \mathcal ( | \alpha - \beta|^{\infty}  ), \\ 
z = x - i \xi, \ \ w = x' - i \xi' , \ \ 
\alpha = ( x, \xi ) , \ \ \beta = ( x', \xi' ) . \end{gathered}
\end{gather} 
Of course this is satisfied in our explicit construction. Note that 
\eqref{eq:eiktr0} determines $ a_k ( \alpha, \beta ) $ uniquely from 
$ a_k ( \alpha, \alpha ) $ modulo $ \mathcal O ( |\alpha - \beta |^\infty ) $.
We can think about constructing $ a_k$'s as follows: define
\begin{equation}
\label{eq:Deltab}  \overline 
\Delta := \{ ( z , \bar z ) : z \in \CC^n \}
\subset \CC^n \times \CC^n \simeq T^* \RR^n \times T^* \RR^n , \ \ 
\CC^n \ni z = x - i \xi \in \CC^n , \end{equation}
which is a totally real subspace (see for instance \cite[13.2]{zw}).
With the variables of \eqref{eq:eiktr0},  write $ a_k ( \alpha , \beta ) = 
b_k ( z , \bar w ) $. Then $ b_k ( z , w ) $ is the almost analytic extension of $ b_k ( z, \bar z ) = b_k|_{\overline \Delta} $. 

Concerning $ \psi_0 $, it is uniquely determined from \eqref{eq:eiktr0} when we put
\begin{equation}
\label{eq:psi0d}  
\begin{gathered} ( d_\alpha \psi_0  ) ( \alpha, \alpha ) = \xi dx ,
 \ \ ( d_\beta  \psi_0  ) ( \alpha, \alpha ) = - \xi dx , \ \
\psi_0 ( \alpha, \alpha ) = 0 , \ \ \alpha = ( x, \xi ) .
\end{gathered}  \end{equation}
Note that we could just demand that $ \psi_0 ( \alpha, \alpha ) = 0 $ 
as then the derivative conditions follow from the equations. Conversely,
the derivative conditions determine $ \psi_0 $ up to an additive constant.

We can now compare $ \mathscr P $ to the Toeplitz operator 
\[ \mathscr T_p  := \Pi_0 M_p \Pi_0 , \ \ M_p u ( \alpha ) = p ( \alpha ) u ( \alpha ) .\] 
(We sometimes abuse notation and write $ p $ for $ M_p $ so that
$\mathscr T_p  := \Pi_0 p \Pi_0 $.)

Using the stationary phase method we obtain
\begin{gather*}  \mathscr T_p u ( \alpha ) = \int 
\widetilde K  ( \alpha, \beta ) u ( \beta ) 
d\beta, \ \  K ( \alpha, \beta ) = e^{\frac i h \psi_0 ( \alpha, \beta ) }
\tilde a ( \alpha, \beta ) ,  \\ \tilde a = \tilde a_0 + h \tilde a_1 + \cdots , \ 
\ \ \tilde a_0 ( \alpha , \alpha  ) = p ( \alpha )  . \end{gather*}
But the uniqueness statement means that
\[  a_0 ( \alpha , \beta ) = \tilde a_0 ( \alpha, \beta ) + \mathcal O ( 
| \alpha - \beta|^{\infty} ) . \]
This immediately gives
\begin{theo}
\label{t:1}
Suppose that $ P $ is a pseudodifferential operator \eqref{eq:pseudo1} 
and that $ \mathscr P := T P T^* $,  $\mathscr T_p := \Pi_0 p \Pi_0 $.
Then 
\[  \mathscr P = \mathscr T_p + \mathcal O ( h )_{ \mathscr H \to \mathscr H }, \]
or 
\begin{equation}
\label{eq:cofe0} \langle T P u , T v \rangle_{ L^2 ( T^* \RR^n ) } = 
\langle p T u , T v \rangle_{ L^2 ( T^* \RR^n ) } + \mathcal O ( h ) \| u \|_{L^2 ( \RR^n ) }\| v \|_{L^2 ( \RR^n )  } .\end{equation}
Also, 
\begin{equation}
\label{eq:sja}
T P  = p T  + \mathcal O ( h^{\frac12})_{L^2 ( \RR^n ) \to L^2 ( T^* \RR^n ) }  . 
\end{equation}
\end{theo}
That \eqref{eq:cofe0} holds
was first observed by Cordoba--Fefferman \cite{cof}
while \eqref{eq:sja} is an earlier result of Sj\"ostrand \cite{sja}. (Both  were formulated differently in the original versions and these are the versions from \cite{M} and \cite{Sj96}.) 

In Theorem \ref{t:2} we will see a stronger formulation which (when the
weight is $ 0 $) applies here as well. We note however that when there is no weight we can use \cite[Theorem 13.10]{zw} to obtain an explicit 
$ q $ such that $ \mathscr P = \Pi q \Pi + \mathcal O ( h^\infty )_{ L^2 \to L^2 } $.

\subsection{Geometry of $ \Pi_0 $}
\label{s:geop}
We now revisit \eqref{eq:eiktr0} and \eqref{eq:psi0d} in geometric terms.
The (quadratic) phase $ \psi_0 $ generates a complex (linear) Lagrangian relation
\begin{equation}
\label{eq:Cpsi} \mathscr C := \{ ( \alpha, d_\alpha \psi_0 ( \alpha, \beta ) ; \beta, 
- d_\beta \psi_0 ( \alpha, \beta ) ) : \alpha, \beta \in \CC^{2n} \}
\subset T^* \CC^{2n} \times T^* \CC^{2n} ,  \end{equation}
that is, a linear subspace of (complex) dimension {$4n$} on which the (holomorphic) symplectic form 
\[  \sigma_2 := \pi_L^* \sigma - \pi_R^* \sigma , \ \ 
\sigma :=  d ( \alpha^* d \alpha)  , \ \ 
\pi_L ( \rho , \rho' ) = \rho, \ \ \pi_R ( \rho, \rho' ) = \rho' , \]
vanishes. 
We note that $ \mathscr C \subset S_1 \times S_2 $ where
\begin{equation}
\label{eq:S1S2}  S_1 = \{ \rho \in T^* \CC^{2n} : \zeta_j ( \rho ) = 0 \}, \ \ 
S_2 = \{ \rho \in T^* \CC^{2n} : \bar \zeta_j ( \rho ) = 0 \}, \ \ 
\bar \zeta_j (\rho )  := \overline{ \zeta_j ( \bar \rho ) } .
\end{equation}
That is a geometric version of \eqref{eq:eiktr0}. Since $ \{ \zeta_j, \zeta_k \} 
= 0 \}$, $ S_j $ are involutive of (complex) dimension $ 3n $.
(Note that we have $ \bar\zeta_j $ rather than $ \widetilde \zeta_j $ as we have 
the usual sign switch in the definition of $ \mathscr C $.)
We also identify the symplectic subspace 
\begin{equation}
\label{eq:S1cS2}   \{  ( \rho, \rho ) : \rho \in S_1 \cap S_2 \} 
\subset \mathscr C . \end{equation}
Since 
\[ S_1 \cap S_2 = \{ ( x, \xi, \xi , 0 ) : ( x, \xi ) \in T^* \CC^n \} , \]
\eqref{eq:S1cS2} is a geometric version of \eqref{eq:psi0d}. 

We can present this more abstractly without an explicit mention of 
$ \zeta_j $'s. Thus we consider 
\begin{equation}
\label{eq:compsymp} V := T^* \CC^{m} ,  \ \ \sigma := \sum_{ j=1}^{m} d z_j^* \wedge 
d z_j , \ \ ( z , z^*) \in T^* \CC^{m} . \end{equation}
For a linear subspace of $ W \subset V $ we define the symplectic annihilator of $ W $ by 
\[ W^\sigma := \{ \rho \in V : \sigma ( \rho, V ) = 0 \}. \]
We then consider {\em involutive} subspaces of $ V $:
\begin{equation}
\label{eq:Sinv}  S \subset V , \ \ S^\sigma \subset S , \ \ \dim_{\CC } S =  2m - k . \end{equation}
The Hamiltonian foliation of $ S $ is defined by the projection
\begin{equation}
\label{eq:firstp} p : S \longrightarrow S/ S^\sigma . \end{equation}

Assume now $ S_1 $ and $ S_2 $ are two such subspaces and that 
\begin{equation}
\label{eq:S1S2tr}   \dim_{\CC} ( S_1 \cap S_2 ) = 2m - 2k , \ \  ( S_1 \cap 
S_1 )^\sigma \cap S_1 \cap S_2 = \{ 0 \} . \end{equation}
This means that $ S_1 $ and $ S_2 $ intersect transversally at a symplectic subspace and that the
 (affine) leaves of the Hamiltonian foliations through points of $ S_1 \cap S_2 $ 
also intersect transversally and we have identifications 
\[  S_1 \cap S_2 \ni \rho \longmapsto \rho + S_j^\sigma \in S_j / S_j^\sigma . \]
Composing the inverse of this map with \eqref{eq:firstp} we obtain complex linear maps
\begin{equation}
\label{eq:secondp}
\begin{gathered}
 p_j : S_j \to S_1 \cap S_2 , \\
  p_1^{-1} ( \rho ) \cap p_2^{ -1} ( \rho ) =\{ \rho \} , \ \ \rho \in 
S_1 \cap S_2 , \ \ \dim p_j^{-1} (\rho) = k.
\end{gathered}
\end{equation}

The abstract (linear) version of \eqref{eq:Cpsi} is given in 
\begin{lemm}
\label{l:relC}
Suppose that two involutive complex subspaces $ S_1$ and $ S_2 $ satisfy
\eqref{eq:S1S2tr} and that $ \mathscr C \subset S_1 \times S_2 $ is
a complex Lagrangian subspace of $ V \times V $. Then the following 
conditions are equivalent:
\begin{enumerate}
\item $ \mathscr C \circ \mathscr C = \mathscr C $,  $ \mathscr C \cap 
\left( ( S_1 \cap S_2 ) \times ( S_1 \cap S_2 ) \right) = 
\Delta ( S_1 \cap S_2 ) $; 
\item $ \mathscr C \cap 
\left( ( S_1 \cap S_2 ) \times ( S_1 \cap S_2 ) \right) = 
\Delta ( S_1 \cap S_2 ) $; 
\item $ \mathscr C := \{ ( \rho_1, \rho_2) \in S_1 \times S_2 : 
p_1 ( \rho_1 ) = p_2 ( \rho_2 ) \}   $,
\end{enumerate}
where $ p_j $ are defined in \eqref{eq:secondp} and, for $ W \subset V $, 
$ \Delta ( W ) := \{ ( \rho, \rho) : \rho \in W \} \subset V\times V$.
\end{lemm} 
\begin{proof} 
We can find defining functions of $ S_i$'s, $ \zeta_j^i $, $ j = 1, \cdots, k $, 
$ \{ \zeta_j^i, \zeta_\ell^i \} = 0 $, (chosen them globally here as we are in the linear case). Then $ H_{\zeta_j^i} $ are tangent to 
$ S_j $. 
$ H_{\zeta_j^1 } \oplus 0 $ and $ 0 \oplus H_{ \zeta_j^2 } $ are tangent to $ \mathscr C $. Defining 
$ \Phi^t_i := \exp ( t_1 H_{ \zeta_1^i} ) \cdots \exp ( t_k H_{ \zeta_k^i } ) $ 
we see that $ \mathscr C $ is invariant under the action of $ \Phi_t^1 \oplus \Phi_s^2 $, $ t, s \in \CC^k $. It then follows that
\begin{equation}
\label{eq:p1p2C} ( \rho_1 , \rho_2 ) \in \mathscr C \ \Longleftrightarrow \ 
( p_1 ( \rho_1 ) , p_2 ( \rho_2 ) ) \in \mathscr C . \end{equation}
With this in place the lemma is immediate:
(1) $ \Rightarrow $ (2) is obvious from the second condition in 
(1);
(2) $ \Rightarrow $ (3) follows from \eqref{eq:p1p2C} and 
dimension counting;
(3) $ \Rightarrow $ (1) is clear: $ ( \rho_1 , \rho ) \in \mathscr C $ and
$ ( \rho, \rho_2 ) \in \mathscr C $ implies that $ \rho  \in S_1 \cap 
S_2 $ 
and hence $  p_1 ( \rho_1 ) = p_2 ( \rho ) = \rho = p_1 ( \rho ) =  p_2 ( \rho_2 ) $. 
\end{proof}

\section{Projector with weights}
\label{s:we}

We now want to prove an analogue of \eqref{eq:cofe0} in the case of 
weighted spaces. For that we assume that $ P = p^w ( x, h D ) $ where
$ p \in S ( 1 )$  has a bounded analytic continuation to a fixed neighbourhood of $ T^* \RR^n \subset \CC^{2n} $. In that case, following Martinez \cite{M} and earlier works, we will show that for $ \varphi \in \CIc ( T^* \RR^n ) $ with $ 
\| \nabla \varphi \|_{L^\infty } $ sufficiently small (depending on the neighbourhood in which $ p $ is analytic) we have
\begin{equation}
\label{eq:cofe1}
\begin{split} \langle  T P u ,  T v \rangle_{L^2_\varphi  } & =
\langle p_\varphi  T u ,  Tv  \rangle_{L^2_\varphi } 
+ 
\mathcal O ( h ) \|  Tu \|_{L^2_\varphi }  
\|  Tv \|_{L^2_\varphi } , \end{split}
\end{equation}
where
\[ p_\varphi ( x, \xi ) := p ( x +  {2}\partial_z \varphi, \xi -  {2}i \partial_z \varphi ), \ \ z = x - i \xi , \ \ 
L^2_\varphi := L^2 ( T^* \RR^n ; e^{-{2}\varphi/h } d \alpha ), 
 \]
see Theorem \ref{t:2} at the end of \S \ref{s:psew}.
To do this 
we follow the same strategy as in \S \ref{s:FBI} and construct a self-adjoint projection 
\begin{gather}
\label{eq:Piph} 
\begin{gathered}  \Pi_\varphi : L^2 ( T^* \RR^n ) \to \mathscr H , \ \  
\Pi_\varphi^2 = \Pi_\varphi , \ \  \Pi_\varphi |_{\mathscr H} = I_{\mathscr H}, \\
\langle \Pi_\varphi u , v \rangle_{ L^2_\varphi }  
= \langle u , \Pi_\varphi v \rangle_{ L^2_\varphi } . 
\end{gathered}
\end{gather}
We write the last statement as $ \Pi_\varphi^{*,\varphi} = \Pi_\varphi $. In what
follows, for the sake of clarity we drop $ \varphi $ and, unless specifically stated, consider the adjoint in 
$ L^2 ( e^{- 2\varphi/h } d \alpha )$ only. 

{\bf To describe $ \Pi_\varphi $ we make the assumption that $ \| \varphi \|_{ C^2 } $ is 
sufficiently small.}

The strategy for describing $ \Pi_\varphi $ as $ h \to 0 $ goes back to the works of Boutet de Monvel--Sj\"ostrand \cite{BS}, Boutet de Monvel--Guillemin \cite{BG}, Helffer--Sj\"ostrand and was outlined for compact manifolds and compactly supported weights in \cite{Sj96}. The argument proceeds in the following
steps:

\begin{itemize}

\item construction of a uniformly bounded operator (as $ h \to 0 $) $ B : L^2_\varphi \to L^2_\varphi $ such that $ Z_j B = \mathcal O ( h^\infty)_{L^2_\varphi \to L^2 _\varphi}  $,
$ B^{*} = B $ and
$ B^2 = B + \mathcal O ( h^\infty )_{ L^2_\varphi \to L^2_\varphi }$;

\item characterization of the {\em unique} properties of the Schwartz kernel of 
$ B $: uniqueness of the phase and the determination of the amplitude from its restriction to the diagonal;

\item finding a projector $ P = \mathcal O ( 1)_{ L^2_\varphi \to 
L^2_\varphi } $ onto the image of $ T $. 

\item choosing $ f \in S ( 1 )  $, $ f \geq 1/C $ so that 
$ A := P M_f P^* $ (in the notation of \S \ref{s:FBI}), satisfies 
$ A = B + \mathcal O ( h^\infty )_{ L^2_\varphi \to L^2_\varphi }$; this relies on the uniqueness properties in the construction of $ B $;

\item expressing $ \Pi $ as a suitable contour integral of the resolvent of $ A $ 
and using it to show that $ \Pi = B + \mathcal O ( h^\infty )_{ L^2_\varphi \to L^2_\varphi }$.
(this, elementary and elegant part, can be copied verbatim from \cite{Sj96}).
\end{itemize}

To construct $ B $ we postulate an ansatz
\[  B u ( \alpha ) = h^{-n} \int e^{ i  \psi ( \alpha, \beta )/h -   {2}\varphi ( \beta )/h } a ( \alpha, \beta ) u ( \beta ) d \beta , \]
and as in \eqref{eq:Kab0} 
\begin{equation}
\label{eq:Kab010} \begin{gathered} e^{ - i \psi ( \alpha, \beta) /h }  Z_j ( \alpha, h D_\alpha ) \left( e^{ i\psi ( \alpha, \beta )/h }
a ( \alpha, \beta ) \right) = \mathcal O ( |\alpha - \beta |^{\infty}  )  , \\
e^{ - i \psi ( \alpha, \beta) /h } \widetilde Z_j ( \beta, h D_\beta ) \left( e^{ i \psi ( \alpha, \beta )/h }
a ( \alpha, \beta ) \right)  = \mathcal O ( |\alpha - \beta |^{\infty}  ) , \\
Z_j := hD_{x_j} - \xi_j - i hD_{\xi_j} , \ \ 
\widetilde Z_j :=   - hD_{x_j} - \xi_j - i h D_{\xi_j} . 
\end{gathered} \end{equation}
We note that for $ \bar Z_k ( \alpha, h D_\alpha ) := \widetilde 
Z_k ( \alpha, - hD_\alpha ) $ we have 
$  ( i/h) [ Z_j , \bar Z_k ] =  -{2  i \delta_{jk} } $.

{Self adjointness of $B$ implies that} we should also have 
\[  \psi ( \alpha, \beta ) = - \overline{ \psi ( \beta, \alpha ) } , \ \ 
a ( \alpha, \beta ) = \overline{ a ( \beta, \alpha ) } , \]
which is consistent with \eqref{eq:Kab0}. 

The fact that the weights to do {not} appear in $ \widetilde Z_j $ may seem surprising but is easily verified: put $ K_B := e^{ i \psi/h} a/h^n $
and note that 
\[ (Z_j ) ^* = e^{  {2} \varphi /h } \bar Z_j 
e^{ -   {2}\varphi/h} , \ \ 
\bar Z_j  v := \overline { Z_j ^t \bar v } .\]
 Then
 \begin{equation}
\label{eq:Kab1} \begin{split}   0 & \equiv ( Z_j  B )^* u ( \alpha ) = B^* Z_j  ^* u ( \alpha ) = B Z_j^* u ( \alpha ) \\ & 
 = \int K_B ( \alpha, \beta ) e^{ -  {2} \varphi ( \beta)/h }(Z_j  )^* u ( \beta )
d \beta \\
& =  \int K_B ( \alpha, \beta ) \bar Z_j ( \beta, h D_\beta ) \left( 
e^{ -   {2}\varphi( \beta)/ h } u ( \beta ) \right) d \beta  \\
& = 
\int \widetilde Z_j ( \beta, h D_\beta ) K_B ( \alpha, \beta ) e^{ - {2}\varphi( \beta ) /h } d \beta, \ \ \ 
\widetilde Z_j v := \overline { Z_j  \, \bar v } .
\end{split}
\end{equation}

Going back to \eqref{eq:Kab010} we obtain simple eikonal and transport equations
for $ \psi $ and $ a $:
\begin{equation}
\label{eq:psieik} 
\begin{gathered} \psi_{x_j} - \xi_j - i \psi_{\xi_j} = \mathcal O ( | \alpha - \beta |^\infty), \  \ 
- \psi_{y_j} - \eta_j - i \psi_{\eta_j} = \mathcal O ( | \alpha - \beta |^\infty),  \\ \alpha = ( x, \xi) , \  \ \beta = ( y, \eta ) , \end{gathered}
\end{equation}
and
\begin{equation}
\label{eq:transa} 
\begin{gathered}  a = a_0 + h a_1 + \cdots, \\ 
\bar \partial_z a_k = \mathcal O ( | \alpha - \beta|^\infty ) , \ \ \partial_w a_k  = \mathcal O ( | \alpha - \beta|^\infty ) , \\  z = x - i \xi, \ w = x' - i \xi' .
\end{gathered}
\end{equation}
To guarantee the boundedness on $ L^2_\varphi $ and decay away from the diagonal we also 
demand that
 \begin{equation} 
\label{eq:psipos} - \Im \psi ( \alpha, \beta )  -  \varphi ( \alpha ) - \varphi ( \beta )  = c_0 | \alpha - \beta |^2 +
\mathcal O ( | \alpha - \beta |^3 ) ,  
\ \  c_0 > 0 . 
\end{equation}

\subsection{Phase construction}
\label{s:phs}

We now need to discuss the ``initial conditions" for $ \psi $: in the
free case they were given in \eqref{eq:psi0d} and geometrically in 
Lemma \ref{l:relC}. We start in an ``ad hoc" way and then move to the geometric version. Thus we require that
\begin{equation}
\label{eq:psidig}
  \psi ( \alpha , {\alpha} ) = -  {2}i\varphi ( \alpha ) . 
\end{equation}
In the notation of \S \ref{s:FBI} (specifically with $\psi_0$ as in~\eqref{eq:Pi0}) we put
\[  \psi ( \alpha, \beta ) = \psi_0 ( \alpha, \beta ) +  \widetilde 
\psi ( z, \bar w ) ,  \]
so that the equations become
\[  \partial_{\bar z } \widetilde \psi ( z, \bar w )  , \ \ \partial_{ w } \widetilde \psi ( z, \bar w )  = \mathcal O (| z - w |^\infty), \ \ 
\widetilde \psi ( z, \bar z ) =  -  {2}i \varphi ( z ) .\]
(Recall that $ \psi_0 ( \alpha, \alpha ) = 0 $ and that
$  {2}\partial_{\bar z }  \psi_0 = \xi_j $, $  {2} \partial_{ w}  \psi_0 =  {\xi'_j} $, $ z = x - i \xi $, $ w = x' - i \xi' $, $ \alpha = ( x,\xi) $, 
$ \beta = ( x', \xi' ) $.)

We note that the analogue
of \eqref{eq:psi0d} is 
\begin{equation}
\label{eq:dalpsi} ( \partial_z \widetilde \psi) ( z, \bar z ) = 
- {2}i \partial_z \varphi ( z ) , \ \ 
( \partial_{w} \widetilde \psi ) ( z, \bar z ) = - {2} i \partial_{\bar z } \varphi ( z ) . 
\end{equation}
This is solved by taking an almost analytic extension of 
$ \widetilde \varphi  = \widetilde \psi|_{ \bar{\Delta } } $ from the totally 
real submanifold $ \bar \Delta $ -- see \eqref{eq:Deltab}. We note
here that $ d ( ( z, w ) , \bar{\Delta } ) = | z - \bar w | $. 

\noindent
{\bf Remark.} 
Note for any smooth function $f(z)$, if $g(z,w)$ is almost analytic near $\bar{\Delta}$ with $g(z,\bar{z})=f(z)$, then since $\partial_{\bar{z}}g(z,w),\,\partial_{\bar{w}}g(z,w)=O(|z-\bar{w}|^\infty)$ we have $\partial_zf(z)=\partial_z g|_{\bar{\Delta}}+\partial_{\bar{w}}g|_{\bar{\Delta}}$. Hence, $\partial_zg|_{\bar{\Delta}}=\partial_z f$. Similarly, $\partial_{\bar{z}}f=\partial_{w}g|_{\bar{\Delta}}$. 
\qed

We then get
\[  \psi ( \alpha, \beta ) = \psi_0 ( \alpha, \beta ) + \widetilde \psi ( z, \bar w ) . \]
Near the diagonal we have
\[ \begin{split} \Im \widetilde \psi ( z, \bar w ) & = 
- {2} \varphi ( z ) + 
\Im \left( ( \partial_w \widetilde \psi)( z , \bar z ) ( \bar w - \bar z ) \right) + \mathcal O ( \| \varphi\|_{ C^2 } | w - z|^2 )\\
&  = 
-  {2}\varphi ( z ) + \Im ( -  {2}i \partial_{\bar z } \varphi ( z ) 
( \bar w - \bar z ) ) + \mathcal O ( {\| \varphi\|_{C^2} }  | w - z|^2 ) . 
\end{split} \]
Similarly, 
\[ \begin{split} \Im \widetilde \psi ( z, \bar w ) & = 
-  {2}\varphi ( w ) + 
\Im \left( ( \partial_z \widetilde \psi)( w , \bar w ) ( z - w ) \right) + \mathcal O ( {\| \varphi\|_{C^2} }  | w - z|^2 )
\\
& = -  {2}\varphi ( w ) + 
\Im \left( ( \partial_z \widetilde \psi)( z , \bar z ) ( z - w ) \right) + \mathcal O ( {\| \varphi\|_{ C^2 } }   | w - z|^2 ) \\
&  = 
-  {2}\varphi ( z ) + \Im ( -  {2}i \partial_{z } \varphi ( z ) 
( z - w ) ) + \mathcal O ( {\| \varphi\|_{C^2} }  | w - z|^2 ) . 
\end{split} \]
Adding up the two equalities we obtain
\[ \Im \widetilde \psi ( z, \bar w )  
= - \varphi ( z ) -  \varphi ( w ) + 
\mathcal O ( {\| \varphi\|_{C^2} }  | w - z |^2 ) .
 \]
Hence, 
\[\begin{split} - \Im \psi ( \alpha, \beta )  -  \varphi ( \alpha ) - \varphi ( \beta )  & = - \Im \psi_0 ( \alpha, \beta ) + 
\mathcal O ( {\| \varphi\|_{C^2} } | \alpha - \beta|^2 )  \\
& =
- \tfrac14 | \alpha - \beta |^2 + \mathcal O ( { \| \varphi\|_{C^2} }  | \alpha - \beta|^2 ) .  
\end{split} \]
Hence, if $ \| \varphi \|_{ C^2} $ is small enough, we obtain 
\eqref{eq:psipos}.

\noindent
{\bf Remark.} A more careful analysis of the quadratic terms would show that 
we only need $ \sum_{ i,j} \partial_{ z_j \bar z_k} \varphi ( z) 
\zeta_j \bar \zeta_k > - \frac 14 | \zeta|^2 $ which is a subharmonicity condition. We will not pursue this direction here.

We now discuss the property
 $  B = B^2 + \mathcal O ( h^\infty )_{ L^2_\varphi \to L^2_\varphi } $
 and that will lead naturally to the construction of $ a $ in 
 \eqref{eq:transa}.
Denoting the kernel of $ B^2 $ by $ K_{B^2} $ (analogue of $ K_B $ in 
\eqref{eq:Kab1}) we have
 \begin{align} K_{B^2}  ( \alpha, \beta ) & = 
\int K_B ( \alpha, \gamma ) K_B ( \gamma, \beta ) e^{ - {2}\varphi( \gamma){/h} } 
d \gamma \notag \\
& = 
h^{-2n} \int e^{ \frac i h ( \psi ( \alpha, \gamma ) + \psi ( \gamma, \beta ) +  {2}i \varphi ( \gamma ) ) 
) } a ( \alpha, \gamma ) a ( \gamma, \beta ) d \gamma . \label{eq:KB2}
\end{align} 
In view of \eqref{eq:psipos} we can assume that $ a $ is supported near
the diagonal and that justifies an application of (complex) stationary
phase. Let
\[ \psi_1 ( \alpha , \beta ) = {\rm{c.v.}}_\gamma \left( \psi ( \alpha, \gamma ) + \psi ( \gamma , \beta ) +  {2}i \varphi ( \gamma ) \right) .
\] 
Since $ B^2 $ is self-adjoint on $ L^2_\varphi $, $ \psi_1 ( \alpha, \beta ) $ satisfies the eikonal equations \eqref{eq:psieik}. 
If we show that \eqref{eq:psidig} holds for $ \psi_1 $ 
then the uniqueness in the construction of $ \psi $ will show that
$ \psi_1 \equiv \psi$ to infinite order on the diagonal.

\noindent
{\bf Remark.} In our special case we can see that $ \psi_1 = \psi $ 
quite immediately. Let us change to holomorphic coordinates {and recall~\eqref{eq:holPhase}. Then, }with 
\begin{equation}
\label{eq:squirrel}
\Phi ( z ) := \tfrac 12 | \Im z |^2 +  \varphi ( z ),\qquad\Psi ( z, w ) = -\tfrac14 ( z - w )^2 + i \widetilde \psi ( z, w ),
\end{equation} 
we have 
{$$
\psi(\alpha,\beta)={i}\left[\Phi_0(z)+\Phi_0(w)\right]-i\Psi(z,\bar{w}), \ \  \Phi_0 ( z ) := | \Im z |^2 .
$$}
{Therefore, with $\gamma\mapsto (v,\bar{v})$, the 
\begin{align*}
i[\psi ( \alpha, \gamma ) + \psi ( \gamma , \beta ) +  {2}i \varphi ( \gamma )]
&=\Psi(z,\bar{v})+\Psi(v,\bar{w})-\Psi(v,\bar{v})- \Phi_0(z) - \Phi_0(w).
\end{align*}}
Then immediately
\begin{equation}
\label{eq:cvPsi}  \Psi ( z, \bar w ) = {\rm{c.v.}}_{v, \bar v } ( \Psi ( z, \bar v ) + 
\Psi ( v , \bar w ) - \Psi ( v, \bar v ) ) .\end{equation}
In fact, treating $ v  $ and $ \bar v $ as independent variables (stationary phase is ``real")
\[ \begin{split} 0 & = \partial_v ( \Psi ( z, \bar v ) + 
\Psi ( v , \bar w ) - \Psi ( v, \bar v ) ) = 
\partial_v \Psi ( v, \bar w ) - \partial_v \Psi ( v, \bar v ) 
\\
&  = 
\partial_{w \bar w} \Phi ( w ) ( \bar w - \bar v ) + 
\mathcal O ( | \bar w - \bar v |^2 )  , \\
0 & = \partial_{\bar v } ( \Psi ( z, \bar v ) + 
\Psi ( v , \bar w ) - \Psi ( v, \bar v ) )
 = 
\partial_{\bar v } \Psi ( z, \bar v ) - \partial_{\bar v}
 \Psi ( v, \bar v ) \\
& = 
\partial_{z  \bar z } \Phi ( z) ( z  -  v ) + 
\mathcal O ( | z - v |^2 ) . 
\end{split} \]
Since $ \partial_{ z \bar z } \Phi $ is non-degenerate we obtain 
$ v = z $ and $ \bar v = \bar w $. Inserting these critical values in 
on the right hand side of \eqref{eq:cvPsi}  yields the desired equality. {In particular,~\eqref{eq:cvPsi} implies that 
$
 \psi(\alpha,\beta)={\rm{c.v.}}_\gamma \left( \psi ( \alpha, \gamma ) + \psi ( \gamma , \beta ) +{2} i \varphi ( \gamma ) \right) .
$}

\subsection{Geometry of the phase}
We now proceed as in \S \ref{s:geop} but with complications due to the
fact that the smooth weight $ \varphi $ makes the problem non-linear and
non-holomorphic.

We start with a formal discussion assuming that 
$ \varphi $ has a holomorphic extension to a neighbourhood of $ 
\CC^{2n}  $, $ U $ (if $ \varphi $ or even $ \nabla \varphi $ are bounded,
everything we have said so far remains valid).  

Let 
\[ \mathscr C := \{ ( \alpha, d_\alpha \psi ( \alpha, \beta ) 
+ {i} d_\alpha \varphi ( \alpha)  ; \beta, - d_\beta \psi ( \alpha, 
\beta ) - {i} d_\beta \varphi ( \beta ) ) : ( \alpha, \beta ) \in 
U \times U \}. \]
We now define 
\[  \zeta_j^\varphi ( \alpha, \alpha^*)  := \zeta_j ( \alpha , \alpha^* + i \partial_\alpha \varphi ( \alpha )  ) , \ \ 
S_1 := \{ \rho \in U : \zeta_j^\varphi ( \rho ) = 0 \}.
\]
We note that (since in our case so far $ \zeta_j ( \alpha, \alpha^* ) $ are linear)
\[  Z_j^\varphi ( \alpha , h D_\alpha ) = 
e^{   \varphi ( \alpha ) / h } Z_j ( \alpha, h D_\alpha ) e^{  {-}\varphi ( \alpha )/ h  } . \]
Using $ \widetilde \zeta_j $ we similarly define $ \widetilde \zeta_j^\varphi $ and $ S_2 $. 

Formally, we are in the situation described in Lemma \ref{l:relC} but 
for $ \varphi \in \CIc ( T^* \RR^n) $ we need an almost analytic version.
In the setting here we already constructed the phase. However, the geometric point of view will be important in \S \ref{s:Laeik} where
we consider a different approach.

\subsection{Amplitude construction}
\label{s:ams}

To find the amplitude $a(\alpha, \beta)$ we once again use the fact that {\em it is enough to determine $a$ on the diagonal}. Application of complex stationary phase to~\eqref{eq:KB2} yields
\begin{equation}
\label{eq:KB21} 
K_{B^2}=h^{-n} e^{\frac{i}{h}\psi(\alpha,\beta)}b(\alpha,\beta)
, \ \ 
b(\alpha,\alpha) \sim \sum_{j}h^jL_{2j}a(\alpha,\gamma)a(\gamma,\alpha)|_{\gamma=\alpha} , 
\end{equation}
where $L_{2j}$ are differential operators of order $2j$ in $\gamma$ and $L_0|_{\Delta}=f(\alpha)$, $ | f ( \alpha ) | > 0 $. Since $\psi(\alpha,\beta)=-\overline{\psi(\beta,\alpha)}$, $f(\alpha)\in\RR$.
We note that if $ a ( \alpha, \beta ) = \overline {a ( \beta  ,\alpha ) }$, then $  b ( \alpha, \beta ) = \overline {b ( \beta  ,\alpha ) }$
as the operator $ B^2 $ is also self-adjoint. In particular, 
$ b ( \alpha, \alpha ) \in \RR $. 

Writing 
$
a\sim \sum_{j}h^ja_j, 
$
we have 
$$
b(\alpha,\beta )\sim \sum_j h^j b_j (\alpha, \beta )  , \ \ 
b_j ( \alpha, \alpha ) = \sum_{k+\ell+m=j}L_{2k}
 a_\ell(\alpha,\gamma)a_{m}(\gamma,\alpha)  |_{\gamma=\alpha} .
$$
We note that if $ a_\ell( \alpha, \beta ) = 
\overline{ a_\ell ( \beta, \alpha ) }$ for $ \ell \leq M $ then 
$ b_\ell |_{\Delta} \in \RR $ for $ \ell \leq M $. Since
\[ b_M ( \alpha, \alpha ) = 2 f ( \alpha ) a_0 ( \alpha, \alpha)
a_M ( \alpha, \alpha ) + \sum_{\substack{ k+\ell+m =M\\\ell , m < M} } 
L_{2k }  a_\ell ( \alpha, \gamma ) a_m ( \gamma, \alpha) |_{ \gamma = \alpha}   , \]
it follows that 
\begin{equation}
\label{eq:aell} a_\ell ( \alpha , \beta )  =  \overline{ a_\ell ( \beta , \alpha ) }
 , \ \ell < M  \ \Longrightarrow \  
 \sum_{\substack{ k+\ell+m =M\\\ell , m < M} } 
L_{2k}  a_\ell ( \alpha, \gamma ) a_m ( \gamma, \alpha) |_{ \gamma = \alpha}  \in \RR .  
 \end{equation}

We iteratively solve the following sequence of equations
\begin{equation}
\label{eq:iterate}
\sum_{k+\ell+m=j}L_{2k}a_\ell(\alpha,\gamma)a_{m}(\gamma,\alpha)|_{\gamma=\alpha}=a_j(\alpha,\alpha)
\end{equation}
with $a_j|_{\Delta}$ real. 
Since $a$ is defined by its values on the diagonal taking almost analytic extensions from $\alpha=\beta$ will complete the proof. First, let
$$
a_0(\alpha,\alpha)=\frac{1}{f(\alpha)}\in C^\infty(T^*\RR^n;\RR)
$$
so that $f(\alpha)a_0(\alpha,\alpha)^2=a_0(\alpha,\alpha)$ (i.e.~\eqref{eq:iterate} is solved for $j=0$). Next, take an almost analytic extension of $a_0|_{\bar{\Delta}}$ to define $a_0$ in a 
small neighbourhood of $\bar{\Delta}$ with $a_0(\alpha,\beta)=\overline{a_0(\beta,\alpha)}$. 

Assume now that~\eqref{eq:iterate} is solved for $j\leq M-1$.  Then,~\eqref{eq:iterate} with $j=M$ reads
\begin{align*}
a_M(\alpha,\alpha)&=\sum_{k+\ell+m=M}L_{2k}a_\ell(\alpha,\gamma)a_{m}(\gamma,\alpha)|_{\gamma=\alpha}\\
&=2a_M(\alpha,\alpha)+\sum_{\substack{k+\ell+m=M\\\ell,m<M}}L_{2k}a_\ell(\alpha,\gamma)a_{m}(\gamma,\alpha)|_{\gamma=\alpha}
\end{align*}
Putting
$
a_M(\alpha,\alpha)=-\sum_{\substack{k+\ell+m=M\\\ell,m<M}}L_{2k}a_\ell(\alpha,\gamma)a_{m}(\gamma,\alpha)|_{\gamma=\alpha}
$ we solve~\eqref{eq:iterate} for $j=M$. From \eqref{eq:aell} we see that
$ a_M ( \alpha, \alpha ) $ is real.
Taking an almost analytic continuation with $a_M(\alpha,\beta)=\overline{a_M(\beta,\alpha)}$ then completes the construction of $a_M$ and hence by induction and the Borel summation lemma we have
\begin{equation}
\label{eq:a2b}
b =a+O(h^\infty)+O(|\alpha-\beta|^\infty).
\end{equation}
with $a(\alpha,\beta)=\overline{a(\beta,\alpha)}$. 

Finally, it remains to check that an operator $R$ with kernel 
$$
K_R(\alpha,\beta)=\chi(|\alpha-\beta|/C)r(\alpha,\beta)e^{\frac{i}{h}\psi(\alpha,\beta)-\frac{ {2}\varphi(\beta)}{h}},\qquad r=O(h^\infty +|\alpha-\beta|^\infty),\, \chi\in C_c^\infty(\RR)
$$
has $R=O(h^\infty)_{L^2_\varphi\to L^2_\varphi}.$ For that, consider the kernel of $e^{-\varphi/2h}Re^{\varphi/2h}$ given by
$$
K_{R,\varphi}(\alpha,\beta)=r(\alpha,\beta)e^{\frac{i}{h}( \psi(\alpha,\beta) + i \varphi(\beta) + i \varphi(\alpha)}.
$$
Now, by~\eqref{eq:psipos}, 
$$
\Big|e^{\frac{i}{h}( \psi(\alpha,\beta) + i \varphi(\beta)+ i\varphi(\alpha)) }\Big|\leq e^{-c|\alpha-\beta|^2/h} 
$$
and hence $K_{R,\varphi}=O(h^\infty)_{C_c^\infty}$ and is supported in $|\alpha-\beta|\leq C$. Schur's lemma then implies $R=O(h^\infty)_{L^2_\varphi\to L^2_\varphi}.$

\subsection{Construction of the projector}
\label{s:cotp} 

We first construct $ P$ with the following properties:
\begin{equation}
\label{eq:ProjP}
P T v = T v , \ \ v \in L^2 ( \RR^n ) , \ \  \| P \|_{L^2_\varphi 
\to L^2_\varphi } \leq C , \end{equation}
with $ C $ independent of $ h $.

The holomorphic structure will be used in the construction of $ P$ 
and we again write $ z = x - i \xi $, $ \Phi_0 ( z ) := \frac 12 |\Im z |^2 $, {and $\Phi$ as in~\eqref{eq:squirrel}}. 
We then recall that 
\[  w = T v , \ v \in L^2 (\RR^n ) \ \Leftrightarrow \ 
u := e^{ | \Im z|^2/2h } w \in L^2_{\Phi_0} ( \CC^n ) , 
\] 
see for instance \cite[\S 13.3]{zw}. We construct
\begin{equation}
\label{eq:ProjPhi}  P_\Phi = \mathcal O ( 1 ) : L^2_\Phi \to H_\Phi , \ \ \  
P_\Phi u  = u , \ \ u \in H_\Phi . \end{equation}
We note that, {since $|\Phi-\Phi_0|\leq C$}, as spaces $ L^2_{\Phi_0} = L^2_{\Phi } $ and the issue is 
the uniform boundedness as $ h \to 0$.
The following $ P_\Phi $ will satisfy \eqref{eq:ProjPhi}:
\begin{equation}
\label{eq:Projphi1}  P_\Phi u ( z ) :=  
\frac{ C^n }{ (\pi h )^n } \int_{\CC^n} e^{ - C| z- w|^2 /h +  
2 \langle z - w , \partial_z \Phi ( z ) \rangle/h } u ( w ) dm ( w ) ,
\end{equation}
provided that $ C $ is suffiently large.
To check uniform boundedness on $ L^2_\Phi $ 
we note that 
\begin{equation}
\label{eq:laPh}  2\Re \langle z - w , \partial_z \Phi ( z ) \rangle = 
 \Phi ( z ) -  \Phi ( w ) + \mathcal O ( \| \Phi'' \|_{L^\infty } |z-w|^2 ) 
.
\end{equation}
Since $ \Phi'' $ is uniformly bounded (in fact constant outside of a compact set) we see that for $ C $ sufficiently large, 
\[ -  \Phi ( z ) +2 \Re \left( - C   | z- w|^2 + 
\langle z - w , \partial_z \Phi ( z ) \rangle \right) + \Phi ( w ) 
\leq - | w - z|^2 , \]
which (using Schur's criterion) shows uniform boundedness of $ P_\Phi $ 
on $ L^2_\Phi $. For $ u \in H_\Phi $ we have $ P_\Phi u = u $
-- see for instance \cite[(13.3.16)]{zw} (the fact that $ \Phi ( z ) $ is not quadratic plays no role in the argument). 
Returning to \eqref{eq:ProjP} we put
\[ P :=  e^{ - |\Im z |^2/2 h } P_\Phi e^{ |\Im z |^2/2h } .\]

We now construct the projector $ \Pi =  \Pi_\varphi $ and relate it
to the parametrix $ B $ constructed above. That is done by repeating the argument presented in \cite{Sj96}. 

We take $ f \in S ( T^* \RR^n ) $, $ f \geq c > 0$, and consider
\[  A := A_f = P M_f P^{*,\varphi}  , \ \  P^{*, \varphi } 
= e^{  {2}\varphi/h } P^* e^{ -  {2}\varphi/h } . \] 
We write the action of $ A $ as follows:
\[ A u ( \alpha ) = \int K_A ( \alpha, \beta ) u ( \beta ) e^{ -  {2}\varphi ( \beta ) / h } d \beta , \ \ 
K_A ( \alpha, \beta ) = h^{-n} e^{ \frac i h \psi_A ( \alpha, \beta ) } a_f ( \alpha, \beta ) , \]
where the phase and amplitude are obtained from the method of stationary phase in the composition defining $ A $.
We claim that 
\[ \psi_A ( \alpha , \beta ) = \psi ( \alpha, \beta ) + \mathcal O ( |\alpha - \beta |^\infty ). \]
To see that we note that \eqref{eq:Kab01}
and hence \eqref{eq:psieik},\eqref{eq:transa} hold with $ \psi $ and $ a $ 
replaced by $ \psi_A $ and $ a_f $. Hence it is sufficient to check that 
\eqref{eq:psidig} holds for $ \psi_A $. We calculate the critical value {on the diagonal}
in notation used in \eqref{eq:Projphi1},\eqref{eq:laPh}: 
\begin{align*} {i\psi_A(\alpha,\alpha)+\Phi_0(z)}&={\rm{c.v.}}_w \left( - 2 C | z - w|^2 + 
2 \Re \langle z - w , \partial \Phi ( z ) \rangle + \Phi ( w ) \right) \\
&=\Phi ( z ) + {\rm{c.v.}}_w \left( - C | z - w|^2 + \mathcal O ( | z - w|^2 ) \right) \\
& =
\Phi ( z ) = \Phi_0 ( z ) + {2} \varphi ( z ) . \end{align*}
Since the left hand side is equal to $ i \psi_A ( \alpha, \alpha ) + 
\Phi_0 ( z ) $, \eqref{eq:psidig}, and hence $ \psi \equiv \psi_A $ follow.

 {Since equations \eqref{eq:transa} are} satisified by $ a_f $, 
$ a_f  $ is determined up to $ O ( h^\infty + |\alpha-\beta|^\infty ) $ by $ a_f|_{\Delta } $. We want to choose $ f \sim \sum_j h^j f_j $ so that
\[  a_f ( \alpha, \alpha ) \sim \sum a_{f,j} ( \alpha , \alpha ) h^j , \
\text{ and } \  a_{f,j} ( \alpha , \alpha ) = a_j ( \alpha, \alpha ) , 
\]
where $ a \sim \sum h^j a_j $ is the amplitude in the construction of $ B $. As in \S \ref{s:ams} (but with different $ L_{2k} $'s, 
$ g := L_0 |_\Delta \neq 0 $) we have, 
\[ a_{f,j} ( \alpha, \alpha ) =
\sum_{ k + \ell = j } L_{2k } f_\ell ( \alpha ) = g ( \alpha ) f_j ( \alpha) + \sum_{\substack{k+ \ell = j\\\ell < j} } L_{2k } f_\ell ( \alpha ). \]
(In our special case, the amplitude in $ P $ is constant which is not the case in generalizations -- but the argument works easily just the same.) Using this, solving $ a_{f,j} ( \alpha ) = a_j ( \alpha ) $ 
for $ f $ is immediate. As in the construction of the amplitude of 
$ B $ in \S \ref{s:ams} we see that $ f $ is real valued and $ f_0 $ is bounded from below.

To summarize, we constructed 
\[  B u ( \alpha ) = \int e^{ \frac i h \psi ( \alpha, \beta ) } 
a ( \alpha, \beta ) e^{ -  {2}\varphi ( \beta ) / h } d \beta \]
and found $ f $ such that
\begin{equation}
\label{eq:propB} 
\begin{gathered}  B = \mathcal O ( 1 ) : L^2_\varphi \to L^2_\varphi, \ \
B = B^{*,\varphi} , \ \ B = B^2 + \mathcal O ( h^\infty)_{ L^2_\varphi 
\to L^2_\varphi } , 
 \\ B = A_f +  \mathcal O ( h^\infty)_{ L^2_\varphi 
\to L^2_\varphi },  \ \ \ A_f := P M_f P^{*, \varphi} , \\
 f ( \alpha ) \sim \sum_j h^j f_j ( \alpha ) \in S ( 1) , \ \ f ( \alpha ) > 1/C .
 \end{gathered} 
 \end{equation}
We can now quote \cite{Sj96} verbatim to see that
\begin{equation}
\label{eq:B2Pi}  \Pi_\varphi = B + \mathcal O ( h^\infty)_{ L^2_\varphi 
\to L^2_\varphi } . \end{equation}
For the sake completeness we recall the argument. To start we observe that for $ u \in \mathscr H := T ( L^2 ( \RR^n ) ) $, 
$ \| u \|_{ L^2_\varphi } > 0 $, 
\[ \begin{split}  \langle A_f u, u \rangle_{ L^2_\varphi}  & = 
\langle P f P^* u, u \rangle_{ L^2_ \varphi } = \langle  f P^* u, P^* u \rangle_{ L^2 \varphi } \geq \min_{\alpha \in T^* \RR^n }  f 
( \alpha ) \| P^* u \|^2_{L^2_\varphi } \\
& \geq 
 \frac{| \langle P^* u , u \rangle |^2}{ C \| u \|^2_{L^2_\varphi} }  =  \| u \|_{L^2_\varphi}^2  /C .  \end{split} 
 \]
Hence, 
\begin{equation}
\label{eq:specAf}
\begin{gathered} \| u \|_{L^2_\varphi } /C  \leq \| A_f u \|_{L^2_\varphi } \leq {C}\| u \|_{L^2_\varphi} , \ \ u \in \mathscr H , \\
A_f u = 0, \ \ u \in \mathscr H^\perp , \ \ A_f^* = A_f, 
\end{gathered}
\end{equation}
and
\begin{equation}
\label{eq:contPi}
\Pi_\varphi = \frac{1}{2 \pi} \int_\gamma ( \lambda - A_f )^{-1} d\lambda,  
\end{equation}
where $ \gamma $ is a positively oriented boundary of
an open set in $ \CC $ containing $ [ 1/C ,  {C} ] $ and excluding $ 0 $.
From \eqref{eq:propB} we know that 
\begin{equation}
\label{eq:AfAf2}  A_f = A_f^2 + \mathcal O ( h^\infty)_{L^2_\varphi \to L^2_\varphi } 
\end{equation}
 and we want to use this property to show that $ \Pi_\varphi $ is close to $ A_f $.
For that we note that if $ A = A^2 $ then, at first for $ |\lambda  | \gg 1 $, 
\[ ( \lambda - A )^{-1} = \sum_{j=0}^\infty  
  \lambda^{-j-1} A^j = \lambda^{-1} + \lambda^{-1} 
  \sum_{ j=0}^\infty \lambda^{-j} A = \lambda^{-1} + A \lambda^{-1} 
  ( \lambda - 1 )^{-1}  .\]
Hence, it is natural to take the right hand side as the approximate inverse in the case when $ A^2 - A $ is small: 
\[ \begin{split} ( \lambda - A_f) ( \lambda^{-1}  +  A_f \lambda^{-1}  ( \lambda - 1)^{-1} )   & = 
I - ( A_f^2 - A_f ) \lambda^{-2}  ( \lambda - 1)^{-1}  
 . \end{split} \]
In view of \eqref{eq:AfAf2} and for $ h$ small enough, the right
hand side is invertible for $ \lambda \in \gamma $ with the inverse equal to $ I + R $, $ R = \mathcal O ( h^\infty)_{L^2_\varphi \to L^2_\varphi } $. Hence for $ \lambda \in \gamma $, 
\[ (\lambda - A_f )^{-1} =  \lambda^{-1} + \lambda^{-1} ( \lambda -1)^{-1} A_f + \mathcal O ( h^\infty)_{L^2_\varphi \to L^2_\varphi } . \]
Inserting this identity into \eqref{eq:contPi} and using Cauchy's formula gives 
\[ \begin{split} \Pi_\varphi & = A_f +  \mathcal O ( h^\infty)_{L^2_\varphi \to L^2_\varphi } = 
B +  \mathcal O ( h^\infty)_{L^2_\varphi \to L^2_\varphi },
\end{split}
\]
which is \eqref{eq:B2Pi}.

\section{Pseudodifferential operators on weighted spaces}
\label{s:psew}

We now want to present the action of pseudodifferential operators 
$ P = p^{\rm{w}} ( x , h D ) $, $ p \in S ( 1 ) $ on the 
FBI transform side.

We will use the notation of \cite[\S 13.4]{zw} and note that {by~\cite[Theorem 13.9]{zw}}
\begin{equation}
\label{eq:TpT} 
\begin{gathered}  T P T^* = e^{ -\Phi_0 ( z) / h } q_{\Phi_0}^w (z, hD_z )  e^{ \Phi_0 ( z ) / h } , \ \ q ( x - i \xi , \xi )  
:= p ( x , \xi )  ,  
\\
q_{\Phi_0}^{\rm{w}} ( z, h D_z ) u :=
\frac{1}{ ( 2 \pi h ) }\int\!\!\!\int_{\Gamma_{\Phi_0} ( z ) } 
q \left(  \frac{ z + w } 2 , \zeta \right) e^{ \frac i h \langle z - w , 
\zeta \rangle } u ( w )  d\zeta \wedge d w , \\ 
 \Phi_0 ( z )  :=  \tfrac 12 | \Im z |^2 , \ \ 
\Gamma_{\Phi_0 } ( z ) : w \mapsto 
\zeta = \frac 2 i \partial_z \Phi_0 \left( \frac{ z + w } 2 \right) , \ \ 
u \in H_{\Phi_0 } .  \end{gathered}
\end{equation}
(See the remark after Lemma \ref{l:ponLph} concerning convergence of the integral.) 
We note here that the correspondence between $ q $ and $ p $ is formally
valid for $ ( x, \xi) \in \CC^{2n} $ and that $ \kappa: ( x, \xi ) 
\mapsto ( x - i \xi, \xi ) $ defines a complex linear canonical transformation. The contour $ \Gamma_{\Phi_0 } $ corresponds to 
integrating $ q|_{ \Lambda_{\Phi_0} } $,  
\[ \Lambda_{\Phi_0} := \kappa ( T^* \RR^n ) = 
\{ ( z , \zeta ) : \zeta = -2 i \partial_z \Phi_0 ( z ) \} .\]

We have the following lemma (see \cite[\S 12.5]{SjR} for a more general version and for applications to scattering resonances):
\begin{lemm}
\label{l:ponLph}
Suppose that $ p $ is holomorphic and bounded on 
$ \RR^{2n } + B_{\CC^{2n } } ( 0 , \rho_0 ) \subset \CC^{2n} $ and that $ \Phi ( z )= \Phi_0 ( z ) + 
 {2}\varphi ( z ) $ with $ \varphi \in \CIc ( \CC^n ) $, $ \| \varphi \|_{ C^2 } $ sufficiently small.
Then on $ H_\Phi = H_{\Phi_0} $, 
\[  q_{\Phi_0 }^{\rm{w}} ( z, h D_z ) = q_{\Phi}^{\rm{w}} ( z , h D_z ) 
= \mathcal O ( 1 ) : H_{\Phi } \to H_{\Phi } , 
\]
where for $ u \in H_{\Phi } $, 
\begin{equation}
\label{eq:qPhi} \begin{gathered}  q_{\Phi}^{\rm{w}} ( z , h D_z ) u  =
\frac{1}{ ( 2 \pi h )^n }\int\!\!\!\int_{\Gamma_{\Phi, c} ( z ) } 
q \left(  \frac{ z + w } 2 , \zeta \right) e^{ \frac i h \langle z - w , 
\zeta \rangle } u ( w )  d\zeta \wedge d w , \\ 
\Gamma_{\Phi , c} ( z ) : w \mapsto 
\zeta = \tfrac 2 i \partial_z \Phi \left( \frac{ z + w } 2 \right) 
+ c i \frac{ \overline {  z - w  } }{ \langle z - w \rangle } ,  
 \end{gathered}\end{equation}
where $ c > 0 $ is sufficiently small.
\end{lemm} 

\noindent
{\bf Remark.} When $ q|_{\Lambda_\Phi} \in \mathscr S ( \Lambda_\Phi ) $ then we can take $ c = 0 $ and have a convergent integral in 
\eqref{eq:qPhi}. Since we assume analyticity the deformed contour provides a quick definition for $ q $ bounded near $ \Lambda_\Phi + 
B_{\CC^n} ( 0 , \rho_0 ) $ which in the case of $ q \in S ( \Lambda_{\Phi} ) $ requires the usual integration by parts and density (of $ \mathscr S \subset S $ in the $ \langle z \rangle^\epsilon S $ topology) arguments -- see the proof of 
\cite[Theorem  13.8]{zw}.

\begin{proof}
We can deform the integral in \eqref{eq:TpT} to the contour given by 
$ \Gamma_{\Phi, c} ( z ) $ (see the remark above concerning convergence): since
$ \varphi $ is small and we take $ c > 0 $ small the deformation is allowed as $ q $ is holomorphic and bounded in $ \Lambda_{\Phi_0 } + 
B_{\CC^n } ( 0 , \rho_0 ) $. 
To see the boundedness on $ H_\Phi $ we use \eqref{eq:laPh}.
\end{proof} 

We now discuss 
\[  \Pi_\varphi T P T^* \Pi_\varphi = \mathcal O ( 1 ) : H_\Phi \to H_\Phi , \]
where the uniform boundedness follows from Lemma \ref{l:ponLph}. 
We can apply the method of stationary phase and for that it is useful to use the notation of \eqref{eq:cvPsi}. The phase then becomes
{\begin{gather*}
\psi(\alpha,\beta)+2i\varphi(\beta)-i[\Psi(z,\bar{w})+\Psi(z,\bar{v})-\Phi(v)+i\langle v-v',\zeta\rangle+\Psi(v',\bar{w})]\\
\zeta = \tfrac 2 i \partial_z \Phi \left( \frac{ v + v' } 2 \right) 
+ c i \frac{ \overline{ v - v' } }{ \langle v - v' \rangle }.
\end{gather*}
We now let
\begin{gather*} \widetilde \Psi  = \Psi ( z, \bar v ) - \Phi ( v ) + i \langle v - v' , \zeta \rangle + 
\Psi ( v' , \bar w ) , \ \ 
\zeta = \tfrac 2 i \partial_z \Phi \left( \frac{ v + v' } 2 \right) 
+ c i \frac{ \overline{ v - v' } }{ \langle v - v' \rangle },
\end{gather*} 
and show that 
\begin{equation}
\label{eq:cvtilPs}  {\rm{c.v}}_{{ v , v', \bar v, \bar v' }}
\widetilde \Psi = 
 \Psi ( z, \bar w ) .  
 \end{equation}  
In fact, for simplicity we take $ c = 0 $ and first compute
\[ \begin{split} 
\partial_{v} \widetilde \Psi & = - \partial_v \Psi ( v, \bar v ) + 
\partial_v \Psi \left( \frac{ v + v' } 2 , \frac{ \bar v + \bar v' } 2  \right) + {\tfrac{1}{2}}\partial_{vv}^2 \Psi \left( \frac{ v + v' } 2 , \frac{ \bar v + \bar v' } 2  \right) ( v - v' ) , \\
\partial_{\bar v'} \widetilde \Psi & ={\tfrac{1}{2}} \partial_{\bar v v }^2 \Psi \left( \frac{ v + v' } 2 , \frac{ \bar v + \bar v' } 2  \right) ( v - v' ).
\end{split} \]
Since $ \partial_{\bar v v } \Psi $ is non-degenerate the second equation 
shows that 
$ v = v' $.  But then the first equation becomes 
$ - \partial_v \Psi ( v , \bar v ) + \partial_v \Psi ( v, 
( \bar v + \bar v')/2 ) = 0 $ so that non-degeneracy of $ \partial_{\bar v v }^2 \Psi $ implies $ \bar v = \bar v' $. 

Computing the remaining two derivatives, 
\[ \begin{split} 
\partial_{\bar v} \widetilde \Psi|_{v = v'}  & =  \partial_{\bar v} \Psi ( z, \bar v ) - \partial_{\bar v } \Psi ( v , \bar v ) , \\
\partial_{ v'} \widetilde \Psi|_{v=v'}   & = 
- \partial_{v } \Psi ( v, \bar v ) + 
 \partial_{ v} \Psi ( v , \bar w ) , \end{split} \]
we use the non-degeneracy of $ \partial_{\bar v v }^2 \Psi $ to see that
$ v = z$ and $ \bar v = \bar w $. But then the critical value of
$ \widetilde \Psi $ is given by $ \Psi ( z , \bar w )$.

We conclude that we have an analogue of \eqref{eq:apw}:
\begin{equation}
\label{eq:Piph1}
\begin{gathered}  \Pi_\varphi T P T^* \Pi_\varphi u ( \alpha ) = 
c_\varphi h^{-n} \int K_{P, \varphi} ( \alpha, \beta ) u ( \beta ) e^{ - {2} \varphi ( \beta)/h} d \beta, \\
K_{P, \varphi} ( \alpha, \beta ) = e^{ \frac i h \psi ( \alpha, \beta ) }
a ( \alpha, \beta ) , \ \ 
a = a_0 + h a_1 + \cdots \\
a_0 ( \alpha, \beta ) = q\left( \frac{ z + w}2, \frac 2 i \partial_z \Phi 
\left( \frac{ z + w} 2 \right) \right) , 
\ \  q ( x  - i \xi , \xi ) = p ( x, \xi ) ,  \\ 
\Phi ( z ) = \tfrac 12 | \Im z |^2 +  \varphi ( z ) , \ \ z = x - i \xi, \ w = y - i \eta, \ \ \alpha = ( x, \xi ) , \ \beta = ( y, \eta ) .
\end{gathered}
\end{equation}
Since 
\[ Z_j ( \alpha , h D_\alpha ) K_P ( \alpha, \beta ) =0 , \ \ 
\widetilde Z_j ( \beta, h D_\beta ) K_P ( \alpha, \beta ) = 0 , \]
construction of $ B $  shows that $ a ( \alpha, \beta ) $ is
determined (modulo $ \mathcal O ( h^\infty )_{L^2_\varphi \to L^2_\varphi } )
$ by $ a|_\Delta $. 
Hence, 
\begin{equation}
\label{eq:t20} \Pi_\varphi T P T^* \Pi_\varphi  = 
\Pi_\varphi M_{p_\varphi } \Pi_\varphi + \mathcal O ( h)_{L^2_\varphi 
\to L^2_\varphi } , \end{equation}
where
\[ \begin{gathered} p_\varphi ( x, \xi ) = q ( z , -i \partial_z \Phi ( z ) ), 
\ z  = x - i \xi , \ \ ( x , \xi ) \in \RR^{2n} \\ 
q ( z ,  \zeta ) = p ( z + i \zeta, \zeta ) , \ \ \Phi ( z ) = \tfrac 12 |\Im z |^2 +  \varphi ( z ) . \end{gathered} \]
But this means that 
\[ \begin{split} p_\varphi ( x, \xi ) & = p ( z + i ( - 2i\partial_z \Phi ( z ) ) , - 2i 
\partial_z \Phi ( z ) ) = p ( x - i \xi + i ( \xi - 2 i \partial_z \varphi ), 
\xi - 2 i \partial_z \varphi ) \\
& = 
p ( x +  {2}\partial_z \varphi , \xi -  {2}i \partial_z \varphi ) , \end{split} \]
which agrees with \eqref{eq:cofe1}. We also obtain the analogue of
\eqref{eq:sja}:
\begin{equation}
\label{eq:sjap} 
T P  = p_\varphi T  + \mathcal O ( h^{\frac12})_{L^2_\varphi ( \RR^n ) \to L^2_\varphi ( T^* \RR^n ) }  .
\end{equation}

We summarize this in the following version of \cite[Corollary 3.5.3]{M}: 
\begin{theo}
\label{t:2}
Suppose that $ P$ is given by \eqref{eq:pseudo1} where the symbol $ p $
enjoys a holomorphic extension satisfying
\begin{equation*} 
| p ( z, \zeta ) | \leq M , \ \ | \Im z |\leq a,  \ \ | \Im \zeta | \leq b .
\end{equation*}
Then for $ \varphi \in C^\infty_{\rm{c}} ( T^* \RR^n ) $ with 
$ \| \varphi \|_{ C^2 } $ sufficiently small and $ L^2_\varphi :=
L^2 ( T^* \RR^n , e^{ - {2}\varphi/h } dx d\xi ) $,
\begin{equation}
\label{eq:t2a}
\langle T P u , T v \rangle_{ L^2_\varphi } 
= \langle M_{P_\varphi} T u , T v \rangle_{ L^2_\varphi } + 
\langle R_\varphi T u , T v \rangle_{ L^2_\varphi } , \end{equation}
where $R_\varphi = \mathcal O ( h^\infty )_{ L^2_\varphi \to L^2_\varphi } $ and 
\[ \begin{gathered} 
 P_\varphi ( x, \xi , h ) = p_\varphi ( x, \xi) + h p_\varphi^1 ( x, \xi ) + \cdots, \\ 
p_\varphi ( x , \xi ) =  p ( x + {2} \partial_z \varphi ( x, \xi)  , \xi - {2} i \partial_z \varphi ( x, \xi) ) , 
 \ \ z = x - i \xi .\end{gathered} \]
\end{theo}
\begin{proof}
The leading term in \eqref{eq:t2a} was already obtained in \eqref{eq:t20}.
Assume that we have obtained $ p_\varphi^j $, $ j = 1, \cdots , J-1$ so
that 
\begin{equation}
\label{eq:PiTPTPi} \Pi_\varphi T P T^* \Pi_\varphi = \Pi_\varphi 
\left(\sum_{j=0}^{J-1} h^j p_\varphi^j \right) \Pi_\varphi + R^J_\varphi, \end{equation}
where 
\[ R^J_\varphi u ( \alpha ) = h^{J -n } c_\varphi \int_{ T^* \RR^n }
e^{ \frac i h \psi ( \alpha, \beta )}  a^J ( \alpha, \beta ) e^{ -  {2}\varphi ( \beta ) /h } {u(\beta)} d \beta , \ \ \  a^J \sim a_0^J + h a_1^J + \cdots,  \]
with $ a^J_k $ satisfying the transport equations \eqref{eq:transa}.
If we apply the method of stationary phase to the first term of the kernel of the first term on 
right hand side of \eqref{eq:PiTPTPi} we obtain a kernel with the expansion
\[   e^{ \frac i h \psi ( \alpha, \beta ) } ( a_0 + h a_1 + \cdots 
+ h^{J-1} a_J + h^J r_0^J + h^{J+1} r_1^J + \cdots ) , \]
where $ a_j$'s are the same as in \eqref{eq:Piph1}. Again all the terms satisfy \eqref{eq:transa} and hence are uniquely determined from their values on the diagonal. Hence, if we put 
\[  p^J_\varphi ( \alpha ) := r_0^J ( \alpha, \alpha ) + a_0^J ( \alpha, 
\alpha ) , \]
we obtain \eqref{eq:PiTPTPi} with $ J $ replaced by $ J+1 $.
\end{proof}

\noindent
{\bf Remark.} The equality \eqref{eq:cofe1} holds for more general weights, $ \varphi \in C^{1,1} $, by more direct arguments -- see
\cite[Theorem 1.2]{SjDuke}. Here we were interested in developing the approach of \cite{HS},\cite{Sj96} based on Bergman-like projectors.

\section{Review of some almost analytic constructions}
\label{s:anex}

In \S \ref{s:projdefo} we will follow \cite{Sj96} and describe the  
orthogonal projector $ L^2_\Lambda \to  T_\Lambda ( L^2 ( \RR^n )) $
(in the notation of Theorem \ref{th:G2ph}). That will involve some 
more involved almost analytic machinery and hence we will first consider
some simpler examples. They seem to be related to some (simpler)
aspects of \cite{Sj74}.

\subsection{General comments about almost analyticity}
\label{s:gaa}

We will be concerned with a neighbourhood of $ \RR^m $ in $ \CC^m $
and for $ U \subset \CC^m $ we define
\[  f \in C^{\rm{aa}} ( U ) \ \Longleftrightarrow \ 
\partial_{\bar z } f ( z ) = \mathcal O_K ( | \Im z |^\infty ) , \ \ 
z \in K \Subset U . \]
This definition is non-trivial only for $ U \cap \RR^m \neq \emptyset $.
We write $ f \sim 0 $ in $ U $ if $ f ( z )  = \mathcal O_K ( | \Im z |^\infty ) $, $ z \in K \Subset U \subset \CC^m $.
We note that (see \cite[Lemma X.2.2]{tre}) that for $ f \in \CI $ that
implies $ \partial^\alpha f \sim 0 $ in $ U $.

Suppose $ \Lambda $ is an almost analytic manifold and 
$ \Lambda \cap \RR^m = \Lambda_\RR $. One way to define $ \Lambda $ 
is through almost analytic defining functions: near any point
$ z_0 \in \Lambda_{{\RR}} $ there exist a neighbourhood $ U $ of $ z_0 $ in 
$ \CC^n $ and $ f_1, \cdots, f_k \in \CI ( \CC^m)$ such that
\begin{gather*} \Lambda \cap U =\{ z : f_j ( z ) = 0 , 1 \leq j \leq n \}, \ \ \partial_z f_j (z_0) \ \text{ are linearly independent,} \\
|\partial_{\bar z} f_j ( z ) | = \mathcal 
O ( | \Im z |^\infty + | \sup_{1 \leq \ell\leq k} f_\ell ( z ) |^\infty ).
\end{gather*}

We now consider {\em almost analytic vector fields}:
\[ V = \sum_{ j=1}^m a_j ( z ) \partial_{z_j} , \ \ a_j \in C^{\rm{aa}} ( \CC^n ) , \]
which we identify with {\em real vector fields} $ \widehat V $
such that for $ u $ holomorphic $ \widehat V f = V $:
\[ \begin{split} 
\widehat V & := V + \bar V = \Re V \\
& =  \sum_{ j=1}^m  \Re a_j ( z ) ( \partial_{z_j} +
\partial_{\bar z_j } ) + i \Im a_j ( z ) ( \partial_{z_j} -
\partial_{\bar z_j } )  \\
& = \sum_{ j=1}^m  \Re a_j ( z ) \partial_{\Re z_j} +  \Im a_j ( z )  \partial_{\Im z_j} . \end{split} \]

\noindent
{\bf Example.} Suppose $ M \subset \CC^m $, $ \dim_{\RR} M = 2k $ is almost analytic. Then 
vector fields tangent to $ M $ are spanned by almost analytic vector fields, $ V_j = a_j ( z) \cdot \partial_z $, $ \partial_{\bar z }
a_j ( z )  = \mathcal O ( |\Im z |^\infty )$, $ z \in M $, $ j = 1, 
\cdots k $. In fact, using \cite[Theorem 1. {4}, 3$^{\circ}$]{mess} 
we can write  {$M$ locally near any $z\in M\cap \RR^m$} as $ \{ ( z' , h ( z' ) ) : z' \in 
\CC^k \} $, $ h = (h_{k+1}, \cdots , h_m ) : \CC^k \to \CC^{m-k} $, $ \partial_{\bar z } h =
\mathcal O ( |\Im z'|^\infty + |\Im h ( z' ) |^\infty ) $. We then 
put
\begin{equation}
\label{eq:Vj}  V_j = \partial_{z_j} + \sum_{ \ell = k+1}^m \partial_{z_j} h_\ell (z') \partial_{z_\ell} .
\end{equation}
The real vector fields $ \widehat V_j $ then span vector fields tangent to $M $. \qed

\medskip

Following \cite{mess} and \cite{Sj74} we define the (small complex time)
flow of $ V $ as follows for $ s \in \CC $, 
$ |s | \leq \delta $ 
\begin{equation} 
\label{eq:defesV}  \Phi_s  (z ) := \exp { \widehat {sV } } ( z ) .
\end{equation}
The right hand side is the flow out at time $ 1 $ of the real 
vector field $ \widehat{sV} $. Unless the coefficients in $ V $ are
holomorphic $ [ \widehat V, \widehat { i V } ] \neq 0 $ which means
that $ \exp ( s + t ) V \neq \exp sV \exp t V $ for $ s, t \in \CC $.
However, we have $ [ \widehat { i V } , \widehat V ] \sim 0 $.

\begin{lemm}
\label{l:flow1}
Suppose that $ \Gamma \in \CC^m $ is an embedded almost 
analytic submanifold and $ V $ is an almost analytic vector field. 
Assume that, 
\begin{equation}
\label{eq:ViVin} \text{ $ \widehat V $, $ \widehat {iV } $ are linearly independent and their span is transversal to $ \Gamma $,} \end{equation}
and that, 
in the notation of \eqref{eq:defesV},
\begin{equation}
\label{eq:assimP}
| \Im \Phi_t ( z ) | \geq |t|/C_K , \ \ z \in K \Subset \Gamma . 
\end{equation}
Then 
for  any $ U \Subset \CC^m $, there exists $ \delta $ such that 
\[ \Lambda  := \left\{ \exp \widehat {t V} ( \rho ) 
: \rho \in \Gamma \cap U  , \ \ |t| < \delta  , \ t \in \CC \right\}  \]
is an almost analytic manifold, $ \Lambda_\RR = \Gamma_\RR $ and  
$ \dim_{\Re} \Lambda = 2k +2$. \end{lemm} 

We will use the following geometric lemma:
\begin{lemm}
\label{l:geof}
Suppose $ Z_j \in \CI ( \RR^m ; T^* \RR^m ) $, $ j = 1, \cdots , J $, are smooth vector fields
and, for $ s \in \RR^J $, 
\[ \langle s, Z \rangle := \sum_{j=1}^J s_j Z_j \in \CI ( \RR^m ; T^* \RR^m ) .\]
Then for $ f \in \CI ( \RR^m ) $
\begin{equation}
\label{eq:pullbyexp}
f ( e^{ \langle s , Z \rangle} ( \rho ) ) =
\sum_{ p=1}^{P} \frac{1}{p!} ( \langle s, Z \rangle)^k f ( \rho ) + 
\mathcal O_K ( |s|^{P+1} ), \ \ \rho \in K \Subset \RR^m . 
\end{equation}
while for $ Y \in \CI ( \RR^m ; T^* \RR^m ) $, 
\begin{equation}
\label{eq:pushbyexp}
e^{ \langle s , Z \rangle }_* Y ( \rho ) = 
\sum_{p=1}
^{P} \frac{1}{p!} \ad_{ \langle s, Z \rangle}^k Y ( \rho ) + 
\mathcal O_K ( |s|^{P+1} ), \ \ \rho \in K \Subset \RR^m .
\end{equation}
\end{lemm}
For a proof see for instance \cite[Appendix A]{fred}. We recall that
$ F_* Y ( F ( \rho ) ) := dF ( \rho ) Y ( \rho ) $.


\begin{proof}[Proof of Lemma \ref{l:flow1}] 
Let $ \iota : \Gamma \hookrightarrow \CC^m $ the inclusion
map. Then 
\[  \partial  \exp ( t_1  \widehat { V } + t_2  \, \widehat { i V } )  \circ \iota (\rho ) 
: T_{(0,\rho)} (\RR^2_t \times \Gamma ) \to T_{\rho} \CC^{m} \]
is given by 
$ ( T, X ) \mapsto T_1 \widehat{  V } + T_2 \widehat { i V } + \iota_* X , $
which, thanks  {to} our assumptions, is surjective onto a $ 2k +2 $ (real) dimensional subspace of $ T^* \CC^m $. Hence, by the implicit function theorem $ \Lambda $ is a $ 2k+2 $ dimensional embedded submanifold of $ \CC^m $. 

To fix ideas we start with the simplest case of $
\Gamma = \{ 0 \} \subset \CC^n $. In that case
$ \{ \Lambda = \{ \Phi_t ( 0 ) : t \in \CC , |t| < \delta \}$, and 
from our assumption $ |\Im \Phi_t ( 0 )| \sim |t_1 \widehat V + 
t_2 \widehat { i V } | \sim | t | $.
The tangent space is given by 
\[ T_{ \Phi_t ( 0 ) } \Lambda = \{ \partial_t \Phi_t ( 0 ) 
T + \partial_{\bar t } \Phi_t ( 0 ) \bar T : T \in \CC \}
\subset \CC^2  .\]
If we show that  
\begin{equation}
\label{eq:parttP} \partial_{\bar t }\Phi_t ( 0 ) = \mathcal O ( |t|^\infty) 
\end{equation} then  
$ d (  T_{ \Phi_t ( 0 ) } \Lambda , i  T_{ \Phi_t ( 0 ) } \Lambda ) = 
\mathcal O ( t^\infty ) $ and almost analyticity of $ \Lambda $ 
follows from \cite[Theorem 1.4, 1$^\circ$]{mess}. 
The estimate \eqref{eq:parttP} will follow from showing that for
any  
holomorphic function $ f $, $  {\partial_{t_1}^{\alpha_1}\partial_{t_2}^{\alpha_2}}\partial_{\bar t} f ( \Phi_t ( 0 ) ) |_{t= 0 } = 0 $. But this follows from 
\eqref{eq:pullbyexp} and the fact that $ [ \widehat V , \widehat { i V } ] \sim 0 $ at $ 0$. Indeed, 
\begin{equation}
\label{eq:partf}  \begin{split} 
 {\partial_{t_1}^{\alpha_1}\partial_{t_2}^{\alpha_2}} \partial_{\bar t} f (\Phi_t ( 0 ) ) |_{t=0} & = 
 {\partial_{t_1}^{\alpha_1}\partial_{t_2}^{\alpha_2}}  \partial_{\bar t } \left( \sum_{k=0}^\infty \frac{1}{k!}
\left( t_1 \widehat V + t_2 \widehat { iV } \right)^k f (0 ) \right)|_{t=0} \\
& =    {\partial_{t_1}^{\alpha_1}\partial_{t_2}^{\alpha_2}}
\left( \sum_{k=0}^\infty \frac{1}{k!}
\left( t_1 \widehat V + t_2 \widehat { iV } \right)^k 
( \widehat V + i  \widehat { i V } )f (0 ) \right)|_{t = 0 } 
\\
& = \widehat V^{\alpha_1 } \widehat { i V}^{\alpha_2} ( 
\widehat V + i \, \widehat { i V } ) f ( 0 ) = 
\widehat V^{\alpha_1 } \widehat { i V}^{\alpha_2}  ( V - V ) f ( 0 ) = 0.
\end{split}
\end{equation} 
The fact that $ \widehat V $ and $ \widehat {i V }$ commute to infinite 
order at $ 0 $ was crucial in this calculation. Holomorphy of $ f $
was used to have $ \widehat W f = W f $. 

We now move the general case. For $ z \in \Gamma $, 
$ T_{\Phi_t ( z ) } \Lambda $ is spanned by 
\begin{equation}
\label{eq:TanPht}   \partial_t \Phi_t ( z ) T + \partial_{\bar t} \Phi_t ( z ) \bar T , \ T \in \CC, \ \ 
d \Phi_t ( z) X , \ \  X \in T_z \Gamma . \end{equation}
We can repeat the calculation \eqref{eq:partf} with $ 0 $ replaced by 
$ z $ to see that, using the assumption \eqref{eq:assimP} and
the fact that $ \Im \Phi_t ( z ) = \Im z + \mathcal O ( t )  $, 
\begin{equation} 
\label{eq:bartphi} \partial_{\bar t } \Phi ( z ) = \mathcal O ( |t|^\infty + 
|\Im z |^\infty ) = \mathcal O ( |\Im \Phi_t ( z ) |^\infty ) . 
\end{equation}
To consider $ d \Phi_t ( z ) X = (\Phi_t)_* Y ( \Phi_t ( z ) )  $ we choose a vector field 
tangent to $ \Gamma $, $ Y $, $ Y_c ( z ) = X $. We choose 
\begin{equation}
\label{eq:YcWc} 
 Y_c  = \widehat 
W_c , \ \ \  W_c = \sum_{ j=1}^k c_j V_j , \ \  c \in \CC^k ,
\end{equation}  
a constant coefficient linear combination of vector fields \eqref{eq:Vj}. 
Then $ d \Phi_t ( z ) X = (\Phi_t)_* Y_c ( \Phi_t ( z ) ) $ and we 
want to show that
\begin{equation}
\label{eq:almostca}   c \mapsto (\Phi_t)_* Y_c ( \Phi_t ( z ) )  \ 
\text{ is complex linear modulo errors  
$\mathcal O ( |\Im \Phi_t ( z ) |^\infty ) $.} \end{equation}
 In view of \eqref{eq:TanPht}
that shows that $ d ( T_{ \Phi_t ( z ) } \Lambda , i T_{\Phi_t ( z ) }
\Lambda ) = 
\mathcal O ( |\Im \Phi_t ( z ) |^\infty ) $ and from 
\cite[Theorem 1.4, 1$^\circ$]{mess} we conclude that $ \Lambda $ is
almost analytic.

To establish \eqref{eq:almostca} we use \eqref{eq:pushbyexp}
with $ \langle s , X \rangle = s_1 \widehat V + s_2 \widehat{i V } $, $ s_1 = \Re t$, $ s_2 = \Im t $. Since $ [ \widehat V , 
\widehat {i V} ] \sim 0 $  and $ \widehat V \sim \widehat { i V }/i $
at $ \Im w = 0 $, we see that
\begin{equation}
\label{eq:Phit}  (\Phi_t )_* Y_c ( w ) = \sum_{ p=0}^\infty \frac {t^p} {p!} 
{\ad_{\widehat V}^p W_c } ( w ) + \mathcal O ( |t|^{K+1} +
|\Im w |^\infty ) . \end{equation}
Because of the form of $ W_c $ (see \eqref{eq:Vj} and \eqref{eq:YcWc}) 
\[ \ad_{\widehat V}^p W_c ( w) = \widehat {\ad_{V}^p W_c } ( w ) 
+ \mathcal O ( |\Im w' |^\infty + |\Im h(w')|^\infty ) , \]
and 
\[ c \mapsto  {\ad_{V}^p W_c } ( w )  \ \text{ is complex linear}. \] 
Since $ w = \Phi_t ( z ) $, $ z \in \Gamma $, 
\[ \begin{split}  | \Im w' | + | \Im h ( w') | & = \mathcal O ( |\Im z' | + | \Im h ( z' ) | + |t| ) \\
& =  
\mathcal O  ( | \Im z | + |t | ) = 
\mathcal O  ( | \Im w | + |t | ) = \mathcal O ( | \Im w| ), 
\end{split} \]
since $ |\Im  w| = |\Im \Phi_t ( z ) | \geq |t|/C $. Combining this estimates with \eqref{eq:Phit} gives \eqref{eq:almostca}.
\end{proof}

\subsection{Quasimodes and a positivity condition}
We make the same assumptions on $ p \in 
\mathscr S $ as above but assume in addition that at $ ( x_0, \xi_0 )$, $ p ( x_0, \xi_0 ) = 0 $ and 
$ \{ \Re p , \Im p \} ( x_0 , \xi_0 ) < 0 $. 
We want to show that there exists $ u ( h ) \in \CIc ( \RR^n ) $ such that for $ P = P ( x, \xi , h ) = p ( x, \xi ) + \mathcal O ( h)_{  \mathscr S} $, 
\begin{equation}
\label{eq:uofh} 
P ( x, h  D , h ) u = \mathcal O ( h^\infty )_{L^2 }, \ \ 
\WFh ( u ) = ( x_0 , \xi_0 ) , \ \ \| u \|_{L^2} = 1 . \end{equation}
 (See \cite[12.5]{zw} for a different argument based on a semiclassical adaptation of the construction of
 Duistermaat--Sj\"ostrand.) The assumption that  $ p \in \mathscr S ( \RR^{2n} ) $ is made for convenience only: the construction is (micro)local in phase space.

\subsubsection{Eikonal equation}
\label{ss:eik}
 {Fix $\tilde{p}$ an almost analytic extension of $p$.} We proceed as follows. Assume that $ (x_0, \xi_0 ) = 
( 0 , 0 ) $ and write $ p ( x, \xi ) = a ( x, \xi ) + i b ( x, \xi ) 
+ \mathcal O ( |x|^2+ |\xi|^2 ) $, $ a, b $ real valued and linear. 
Since $ \{ a , b \} = - c^2 < 0 $, the linear version of 
Darboux's theorem \cite[Theorem 21.1.3]{H3} shows that
there exists a {\em linear} symplectic change of variables $ \kappa ( y , \eta) 
= ( x, \xi ) $  {(preserving $T^*\mathbb{R}^n$)} such that 
$$
 \kappa^* a = c\eta_1  +\mathcal O(|\eta|^2+|y|^2),\qquad \kappa^* b = - c y_1 +\mathcal O(|\eta|^2+|y|^2) .$$ 
 We now switch to coordinates $ ( y , \eta ) $
and we denote them again by $ ( x, \xi ) $. Writing $ p $ for 
$ \kappa^* p $  {and $\tilde{p}$ for $\kappa^*\tilde{p}$}, we obtain, 
\begin{equation}
\label{eq:linnof}  p ( 0, 0 ) = 0 ,  \ \  p ( x, \xi ) = c ( \xi_1 - i x_1 ) + 
\mathcal O ( |x|^2+ |\xi|^2 ) . \end{equation}
For $ s \in \CC^{n-1} $, small there exists $ \zeta_1 ( s ) $ such
that 
\[  {\tilde{p}} ( (0 , s ), ( \zeta_1 ( s ), i  s ) ) =0 , \ \ 
\zeta_1 ( 0 ) = 0 , \ \ \partial_s \zeta_1 ( 0 ) = 0 , \ \ 
\ \ \partial_{\bar s } 
\zeta_1 ( s )  = \mathcal O ( |\Im s |^\infty ) , \ \ \alpha > 0 . \ \ 
 \]
We put $ \Lambda_0 := \{ ( ( 0 ,  s ) , ( \zeta_1 ( s ) , i s ) ) \} $
and then, in the notation of \eqref{eq:defesV} we define
\[  \Lambda = \{ \exp \widehat{ t H_{ {\tilde{p}}} }  ( \rho )   : \rho \in \Lambda_0 , \ 
t \in \CC,  \ |t| < \epsilon \} \subset T^* \CC^n .
  \]
To check that $ \Lambda $ is an almost analytic Lagrangian submanifold of
$ T^* \CC^n $ we use Lemma \ref{l:flow1}. The transversality 
condition \eqref{eq:ViVin} follows immediately form \eqref{eq:linnof}
and it remains to 
check \eqref{eq:assimP}. For that we note that with $ t = t_1 + it_2 $
(and recalling that $  {\zeta}_1 ( s ) =  \mathcal O ( |s|^2 ) $, 
\[  \begin{split} \Im \Phi_t ( ( 0  , s ) , ( \zeta_1 ( s )  , i s ) )) & = 
\left(  t_2 c  , 
\alpha \Im  s  , 
 c t_1  , -\Re s  \right) 
+ 
\mathcal O ( |t|^2 + |s|^2 ) . \end{split} \]
Hence, we obtain \eqref{eq:assimP}:  
\[ \begin{split} | \Im \Phi_t ( (  0 ,  s ) , ( \zeta_1 ( s )  , i s ) ))|
& \geq  c( |t_1| + |t_2| ) +   |s| - 
\mathcal O ( |t|^2 + |s|^2 ) \\
&  \geq |t|/C+ {|s|/C} , \ \ \ |s| \ll 1. \end{split} \]

We now claim that
$ \Lambda $ is positive in the sense that for 
\begin{equation}
\label{eq:pos1} \tfrac 1i \sigma ( X , \bar X ) 
\geq c |X|^2  , \ \  X \in T_{(0,0)}^* \Lambda 
\subset T_{(0,0)}^* \CC^n . \end{equation} 
(Here $ \sigma $ is the symplectic form \eqref{eq:compsymp}.) In fact, 
vectors in $ T_{(0,0)}^* \Lambda  $ are given by 
\begin{equation} 
\label{eq:tan2La} X = ( ( T , S) , 
( i T , i S ))  , \ S \in \CC^{n-1}, \ 
T \in \CC,  \end{equation}
from which \eqref{eq:pos1} follows. 

We now note that the (real) linear transformation $ \kappa $
extends to a complex linear transformation on $ \CC^n \times \CC^n $
and we can go back to the original coordinates $ ( x, \xi ) $ 
by taking the almost analytic Lagrangian manifold 
$ \kappa ( \Lambda ) $. We also note that the positivity condition 
\eqref{eq:pos1} is invariant under linear symplectic transformations
which are real when restricted to $ \RR^n \times \RR^n $ (as then 
$ \kappa ( \bar X ) = \overline \kappa ( X ) $). Hence $ \kappa ( \Lambda ) $ is an almost analytic positive Lagrangian and 
we now denote it by $ \Lambda $.

From \eqref{eq:tan2La} we see that $ \pi_* : T_{(0,0)} \Lambda \to 
T_0 \CC^n $ is onto and hence we have an almost analytic generating function, that is $ \Psi ( z) $, 
\[ \partial_{\bar z } \Psi = \mathcal O ( 
|\Im z |^\infty + | \Im \Psi ( z) |^\infty )  \]
 such that, as almost analytic manifolds, 
\begin{equation}
\label{eq:paraLa} \Lambda \sim \{  ( z , \Psi_z ( z ) ) : |z| < \epsilon \}  , \ \ 
\Psi_z ( 0 ) = 0 . \end{equation}
\begin{proof}[Proof of \eqref{eq:paraLa}]
Since $ \Lambda $ is a.a. Lagrangian, we have 
$ \sigma|_\Lambda \sim 0 $ (vanishes to infinite order at $ \Lambda_\RR $) while the projection property shows that,
near $ z = 0 $, $ \Lambda = \{ ( z, \zeta ( z ) ) : z \in \CC^n \}$, 
$ \zeta ( 0 ) = 0 $.
Hence $ d ( \zeta ( z ) d z ) \sim 0 $ and (see \cite[Theorem 1.4, 3$^\circ$]{mess})
\[ \partial_{\bar z } \zeta ( z ) = \mathcal O ( |\Im z |^\infty + 
|\Im \zeta ( z ) |^\infty ) .
\]
We note that for $ z =x \in \RR^n $, the strict positivity at
$ \Lambda_\RR = \{ ( 0 , 0 )\} $ shows that 
\begin{equation}
\label{eq:imzeta}
| x|/C \leq | \Im \zeta ( x ) | \leq C |x|, \ \ x \in \RR^n, \ |x| < 
\epsilon . \end{equation}

We now see that 
\[  0 \sim \sigma|_\Lambda = \sum_{j=1}^n \partial_z \zeta_j ( z ) \wedge 
d z_j + \mathcal O ( | \Im z|^\infty + |\Im \zeta ( z ) |^\infty )_{
C^\infty ( \CC^n ;\wedge^{2n} \CC^n )} , \]
and in view of \eqref{eq:imzeta} 
\[  \partial_{z_k} \zeta_j ( x ) - \partial_{z_j} \zeta_k ( x ) 
= \mathcal O ( | x |^\infty ) , \ \ x \in \RR^n, \ |x| < 
\epsilon . \]
For $ x \in \RR^n $  define $ \Psi $ by the simplest version of the Poincar\'e lemma:
\[  \Psi ( x ) = \int_0^1 \zeta ( t x ) \cdot x dt .\]
Then
 \begin{equation}
\label{eq:aaPoin} \begin{split} \partial_{x_j} \Psi  ( x ) & =  \int_0^1 \left(\sum_{k=1}^n t z_k \partial_{x_j} \zeta_k ( t x )  + \zeta_j ( t x ) \right) dt \\
& =  \int_0^1\left( \sum_{k=1}^n t z_k \partial_{x_k} \zeta_j ( t x )  
+  \zeta_j ( t x ) \right) dt + \mathcal O ( |x|^\infty  ) \\
& = \int_0^1 \partial_t ( t \zeta_j ( t x ) ) dt + 
 \mathcal O ( |x|^\infty )  = \zeta_j ( x ) + \mathcal O ( | \Im \zeta ( x ) | ^\infty) ,
\end{split} 
\end{equation}
in the last argument we used \eqref{eq:imzeta} again. 
We now define $ \Psi ( z ) $ as an almost analytic extension of 
$ \Psi $. From \cite[Proposition 1.7(ii)]{mess} we obtain \eqref{eq:paraLa}. 
\end{proof}

The strict positivity of $ \Lambda $ implies that $ \Im \Psi_{xx} (  0 ) $
is positive definite: 
\begin{gather*}  T_{ (0,0) } \{ z , \Psi_z ( z ) \} = 
\{ ( Z, \Psi_{xx} ( 0 ) Z ) : Z \in \CC^n \}, \\
 \Im \langle \Psi_{xx}( 0 )  Z, \bar Z \rangle  = \tfrac 1 i \sigma ( ( Z, \Psi_{xx} ( 0 ) Z ), ( \bar Z , \overline \Psi_{xx}, Z )) 
 \geq c | Z|^2 . 
 \end{gather*}
 The eikonal equation is satisfied in the following sense:
for $ z \in \CC^n $, $ |z| < \epsilon $, 
\begin{equation}
\label{eq:aaeik}
\widetilde p ( z, \Psi_z ) = \mathcal O ( |\Im z |^\infty + 
|\Im \Psi_z|^\infty ) = \mathcal O ( | \Im z |^\infty  + 
|\Im \Psi |^\infty) , 
\end{equation}
(We can replace $ \Im \Psi_z $ with $ \Im \Psi $ as 
for $ \Im z = 0 $, $ \Im \Psi \geq 0 $ and hence 
$ |\Im \Psi_x | \leq C | \Im \Psi |^{\frac12} $.)

\begin{proof}[Proof of \eqref{eq:aaeik}]
We have for 
 $ s \in \CC $, 
\begin{align*}
 \widehat{s H_{ {\tilde{p}} }} \widetilde p & = 
\overline{s \partial_\zeta \widetilde p} \cdot \partial_{\bar z}\widetilde p - \overline{s \partial_z \widetilde p} \cdot \partial_{\bar \zeta}\widetilde p = \mathcal O 
( | \Im z |^\infty +|  {\Im \zeta}|^\infty ) .
\end{align*}
Since $ \widetilde p |_{  {\Lambda_0}} = 0 $, we see that
$ \widetilde p  {(z,\Psi_z)} \sim 0 $ at $ \Lambda_\RR $.
\end{proof}

Hence to find $ u $ satisfying \eqref{eq:uofh} we take
\begin{equation}
\label{eq:ans4u}  u( x ) := e^{ i \Psi(x)/h } a ( x, h ) . 
\end{equation}
Almost analytic extension of $ a $ will make a natural appearance in the
transport equation. 

\subsubsection{Transport equations} 
\label{ss:te}

We write the amplitude as $ a = a_0 + h a_1 + \cdots $, $ a_j \in S $
and 
to find the transport equations we apply the method of com  plex 
stationary phase \cite[Theorem 2.3, p.148]{mess}  to 
\[ \begin{split} P u ( x ) & = \frac{1}{ ( 2 \pi h )^n }  \int_{ \RR^n } \int_{\RR^n} 
P ( x, \xi , h ) e^{ \frac i h ( \langle x - y , \xi \rangle + 
 \Psi ( y )  ) } a ( y, h ) dy d \xi 
\\
 & =   e^{\frac ih\Psi(x)}\left[\widetilde p ( x, \Psi_x )  a ( x, h )
+ \frac h i \left( \sum_{j=1}^k 
\partial_{\zeta_k} \widetilde p( x, \Psi_x  ) 
\partial _{x_k} + \tfrac 12 \sum_{ k = 1}^n \partial^2_{x_k \xi_k} \widetilde p( x, \Psi_x )  \right) a ( x, h ) \right. \\
& \ \ + \left. h \left(  \widetilde p_1 ( x,  
\Psi_x  )  - i \sum_{k,\ell = 1}^n 
\Psi_{x_k x_\ell} ( x ) \partial_{\xi_k \xi_\ell}^2 \widetilde p ( x, \Psi_x  )  \right) a ( x, h ) + 
\mathcal O ( h^2 )_{\mathscr S }\right].\end{split} 
\]   
The first term is estimated using \eqref{eq:aaeik} and the transport 
equation become
\begin{equation}
\label{eq:traspa}  
\begin{gathered}  V_p \widetilde a_k (z) + \tfrac12 {\rm{div}} V_p \, 
\widetilde a_k ( z) + 
i c_\Psi \widetilde a_k ( z ) = F_{k-1} ( \widetilde a_0, \cdots, \widetilde a_{k-1} ) , \ \ F_{-1} \equiv 0, \\
V_p : = (\pi_\Lambda )_*  H_{  {\tilde{p}}} = \partial_{\zeta_k} \widetilde p( x, \Psi_x  ) \partial _{z_k} , \ \ \pi_\Lambda : \Lambda 
= \{ ( z, \Psi_z (  z) )\} \to \CC^n , \\
c_\Psi ( z  ) := 
\widetilde p_1( z, \Psi_z ) - i \sum_{k,\ell = 1}^n 
\Psi_{z_k z_\ell} ( z ) \partial_{\zeta_k \zeta_\ell}^2 \widetilde p ( z, \Psi_z  ) .
\end{gathered}
\end{equation}
We now solve these equations using the ``almost analytic" flow of 
$ V_p $:
\begin{equation}
\label{eq:propztw} 
\begin{gathered}  z ( t, w') := \exp ( \widehat {t V_p } ) ( 0 , w' ) , \ \ |\Im  { [z ( t , w' )-w']}| \sim |t | , \\
w' \in B_{\CC^{n-1} } ( 0 , \epsilon) ,  
 \ \ t \in \CC, \ \ |t| < \epsilon . 
\end{gathered}
\end{equation}
So, for instance,
\[  \widetilde a_0 ( z  ) := \exp g_0 ( z  ), \ \ 
g_0 ( z ( t, w')) := 
 - \int_0^1 t b_0 ( z ( ts, w' ) )ds, \ \  b_0 := \tfrac12 {\rm{div}} V_p + i c_\Psi . \]
We now calculate the action of $ V_p $ on the $ g_0 ( z ( t,w'))$
using almost analyticity of $ b_0 $ and the properties of 
$ z ( t, w' ) $ in \eqref{eq:propztw}: 
{\begin{align*}
V_pg_0(z(t,w'))&=-\int_0^1 \sum_{k=0}^\infty \frac{s^k}{k!}V_p\widehat{tV_p}^ktb_0(z(0,w'))+\mathcal O(|t|^\infty)ds\\
&=-\int_0^1 \sum_{k=0}^\infty \frac{s^k}{k!}\widehat{tV_p}^{k+1}b_0(z(0,w'))+\mathcal O(|t|^\infty)+\mathcal O(|\Im z(0,w')|^\infty)ds\\
&= -\sum_{k=0}^\infty \frac{1}{(k+1)!}\widehat{tV_p}^{k+1}b_0(z(0,w'))+\mathcal O(|t|^\infty+O(|\Im z(0,w')|^\infty)\\
&= -b_0(z(t,w'))+\mathcal O(|t|^\infty+|\Im z(0,w')|^\infty)\\
&=-b_0(z(t,w'))+\mathcal O(|\Im \Psi|^\infty+|\Im z|^\infty).
\end{align*}}
 This gives \eqref{eq:traspa} with $ k = 0 $.
Similarly we obtain solutions to the remaining transport equations.
We obtain $ a $ by taking an asymptotic sum and multiplying it by 
$ \chi ( x ) $ where $ \chi \in \CIc ( B_{\RR^n } ( 0, \epsilon ) $,
$ \chi \equiv 1 $ near $ 0 $. 
Then returning to \eqref{eq:ans4u} we see that
\begin{equation}
\label{eq:comqm} 
\begin{gathered}
  P ( x, h D_x , h ) ( e^{ i \Psi ( x ) /h } a ( x, h ) ) = 
\mathcal O ( h^\infty +  {e^{-|\Im \Psi(x)|/Ch}}|\Im \Psi ( x ) |^\infty )_{ \CIc } 
= \mathcal O ( h^\infty )_{ \CIc } , \\
\WFh \left( e^{ i \Psi ( x ) /h } a ( x, h ) \right) = \{ (0 , 0 ) \}.
\end{gathered}
\end{equation}

\section{Projector in the case of deformations}
\label{s:projdefo}

We now present a version of \cite[\S 2]{Sj96} in the case of 
the usual FBI transform on $ \RR^n $. It is based on 
deformation of $ T^* \RR^n$ to a I-Lagrangian,  $ \RR$-symplectic submanifold of $ T^* \CC^n $.  In \S \ref{s:wvsd} we will show that this approach, described in \S\S \ref{s:we},\ref{s:psew},
is equivalent to the approach using 
weights.

The FBI transform and weights used in \cite[\S 2]{Sj96} are different 
from the ones used in \cite{M} and \S \ref{s:we}. The procedure of \cite{Sj96}, and earlier of \cite{HS}, involves deformation of $ T u ( x, \xi ) $ to an
I-Lagrangian, $ \RR$-symplectic submanifold of $ \CC^{2n} $:
\begin{equation}
\label{eq:LaG} \Lambda = \Lambda_G := 
\{ ( x + i \partial_\xi G ( x, \xi) , \xi - i \partial_x G ( x, \xi ):  ( x, \xi ) \in 
\RR^{2n} \},\end{equation}
 where $ G \in \CIc ( \RR^{2n} ) $ is assumed to be small in 
$ C^2 $. This means that for the symplectic form \eqref{eq:compsymp} 
on $ \CC^{2n} = T^* \CC^n $ we have
\[ \Im \sigma |_\Lambda = 0 , \ \  \sigma_\Lambda := \Re \sigma |_{\Lambda} \ \text{is non-degenerate}. \]
Smallness of $ G $ is needed for the second property. We also note that
$ \Lambda $ is a (maximally) totally real submanifold of $ \CC^{2n} \simeq \RR^{4n} $,
$ T_\rho \Lambda \cap i T_\rho \Lambda = \{ 0 \} $.

We parametrize $ \Lambda $ by $ ( x, \xi ) $ using \eqref{eq:LaG} and 
define
\begin{equation}
\label{eq:TLa}  T_\Lambda u ( x, \xi ) = T u ( x + i G_\xi ( x, \xi ) , 
\xi - i G_x ( x, \xi ) ) . \end{equation}

A natural weight associated to $ G $ is given by $ H ( x, \xi )$ satisfying
\begin{equation}
\label{eq:defH}   d_{x,\xi}  H = - \Im  \zeta \cdot dz |_{ \Lambda } . 
\end{equation}
Since 
\[  -\Im \zeta \cdot dz |_{\Lambda } = - \Im ( \xi - i G_x ) d ( x + i G_\xi)
= (G_x - (\xi \cdot G_{\xi} )_x  )\cdot dx - ( G_{\xi\xi}\xi) \cdot d\xi , \]
$ H $ is given by (we choose $ H = 0$ for $ G = 0 $)
\begin{equation}
\label{eq:formH} H ( x, \xi ) = G ( x, \xi ) - \xi \cdot G_\xi ( x, \xi) . \end{equation}
\begin{lemm}
\label{l:TSLa}
For $ u \in \mathscr S ( \RR^{2n} )  $ define  
\begin{equation}
\begin{gathered}
\label{eq:SLa}  S_\Lambda u ( y ) := \bar c h^{ -\frac{3 n} 4 } \int_{ \RR^{2n}}
e^{ \frac i h ( \langle y - x - i G_\xi , \xi - i G_x  \rangle + 
\frac i 2 ( x + i G_\xi - y )^2 ) } b ( x, \xi ) u ( x, \xi ) d x d \xi,
\\
b ( x, \xi ) d x \wedge d \xi = d ( \xi - i G_x ) \wedge d ( x + i 
G_\xi ) .
\end{gathered}
\end{equation}
Then 
\begin{equation}
\label{eq:SLaT}  S_\Lambda T_\Lambda  v = v , \ \ v \in L^2 ( \RR^n ) , \end{equation}
and 
\begin{equation}
\label{eq:TLaS} T_\Lambda S_\Lambda = \mathcal O ( 1 ) : L^2_\Lambda \to L^2_\Lambda, 
\ \ \ L^2_\Lambda :=  L^2  ( \RR^{2n}, 
e^{ - 2 H /h } dx d \xi )  . \end{equation}
\end{lemm}

\noindent
{\bf Remark.} The weight $ H $ defined by \eqref{eq:defH} is precisely the unique weight (up to an additive constant) for which \eqref{eq:TLaS} holds -- see \eqref{eq:dImpsi} in the proof below.

\begin{proof}
To prove \eqref{eq:SLaT} we write out the composition and deform the 
contour. The phase in the composition is given by
\[   \langle \xi - i G_x , y - y' \rangle + \tfrac i 2 \left( 
(x + i G_x - y)^2 + ( x + i G_x - y' )^2 \right). \]
If 
$ z = x + iG_\xi ( x, \xi ) $, $ \zeta = \xi - i G_x ( x, \xi ) $, 
then our choice of $ b $ shows that
$  b ( x, \xi ) d x d\xi = d \zeta \wedge d z $ 
and,  by deforming the contour from $ \Lambda_G $ to $ \Lambda_0 := \RR^{2n } $ (note that $ \Lambda_G $ and $ \Lambda_0 $ coincide outside of a compact set), 
\[ \begin{split} S_\Lambda T_\Lambda u ( y ) &= 
c \bar c h^{ - \frac{3n } 2 } \int\!\!\!\int_{\Lambda_G }
e^{\frac i h ( ( y - y') \zeta + \frac i 2 ( ( y - z )^2 + ( y' - z)^2 ))
} u ( y' )  d \zeta dz dy' \\
& = 
c \bar c h^{ - \frac{3n } 2 } \int\!\!\!\int_{\Lambda_0 }
e^{\frac i h ( ( y - y') \zeta + \frac i 2 ( ( y - z )^2 + ( y' - z)^2 ))
} u ( y' )  d \zeta dz dy \\
& = T^* T u ( y ) = u ( y ). 
\end{split} \]
To prove \eqref{eq:TLaS}, we complete the squares in the phase arising in the composition 
of $  T_\Lambda S_\Lambda $ to obtain the phase
\begin{equation}
\label{eq:psiLa} 
\begin{gathered} \psi_\Lambda = 
\tfrac12 ( \langle z , \zeta \rangle -
\langle z' , \zeta' \rangle) + \tfrac 12 ( 
\langle z , \zeta' \rangle - \langle z' , \zeta\rangle ) 
+ \tfrac i 2 ( ( \zeta - \zeta')^2 + ( z - z' )^2 )) , \\
z := x + i G_\xi  ,  \ z' := x' + i G_{\xi'}, \\
\zeta := \xi - i G_{x} , \ \zeta := \xi' - i G_{x'}, \ \
\end{gathered}
\end{equation}
and where $ G_{\bullet'} : = G_{\bullet'} ( x', \xi')$. 

We calculate (noting that as $ \Lambda $ is totally real we can use 
holomorphic differentials by taking almost analytic extensions),
$  d \Im \psi_\Lambda  = 
( \partial + \bar \partial) \Im \psi_\Lambda = 
\Im \partial \psi_\Lambda $,
where $ d $ denote the differential with respect to $ (x,\xi,x',\xi') $
and $ \partial $ the holomorphic differential with respect to $ (z,\zeta,z',\zeta') $.
Using the expression above and restricting to $ z = z' $ and $ \zeta = \zeta' $ we see that
\begin{equation}
\label{eq:dImpsi}  d \Im \psi_\Lambda |_{ (x,\xi ) = ( x',\xi')}   = \Im ( \zeta dz - \zeta dz ' ) =
\left( - d_{x,\xi} {H} + d_{x',\xi'} H \right)|_{ ( x , \xi) = ( x',\xi')} . 
\end{equation}
This means that
\begin{equation}
\label{eq:ImpsiLa}  \Im \psi_\Lambda = - H ( x, \xi ) + \mathcal O ( ( x - x')^2 + 
(\xi - \xi' )^2 ) + H ( x', \xi' ) , \end{equation}
and as $ G $ (and $ H$ ) are small in $ C^2 $ the comparison with the 
case $ G = H =  0 $ gives
\[ \Im \psi_\Lambda = - H ( x , \xi ) + 
( \tfrac12 - \mathcal O ( \| G \|_{C^2} ) ) ( ( \xi - \xi')^2  + 
( x - x')^2 ) + H ( x', \xi' ) . \]
The Schur criterion now gives the boundedness in \eqref{eq:TLaS}.
\end{proof}

\noindent
{\bf Remark.} A more pedestrian way of seeing 
\eqref{eq:ImpsiLa} follows from a direct calculation and from using the formula \eqref{eq:formH}:
\[  \begin{split} 2 \Im \psi_\Lambda & 
= \xi G_\xi - \xi' G_{\xi'} - x G_x + x' G_{x'} 
+ \xi' G_\xi - x G_{x'} + x' G_x - \xi G_{\xi'} \\
&  \ \ \ + 
( \xi - \xi')^2 - ( G_x - G'_x )^2 + 
( x - x')^2 - ( G_\xi - G_{\xi'})^2  \\
& = 2 \xi G_\xi - 2 \xi' G_{\xi} + 2 (\xi' - \xi) G_\xi +
2(x' -x  ) G_{x }  
\\
& \ \ \ 
+ (\xi' - \xi) ( G_{\xi'} - G_\xi) + 
( x' - x ) ( G_{x'} - G_x ) 
 \\
&  \ \ \ 
+ 
( \xi - \xi')^2 - ( G_x - G'_x )^2 + 
( x - x')^2 - ( G_\xi - G_{\xi'})^2  \\
& =  - 2 G ( x, \xi ) + 2 \xi G ( x, \xi) + 2 G  ( x', \xi' ) - 
2 \xi' G_{\xi'} ( x', \xi' ) 
\\
&  \ \ \ + ( 1  - \mathcal O ( \| G \|_{C^2 }  )
( ( \xi - \xi')^2  + 
( x - x')^2 ) \\
& = - 2 H ( x, \xi) + 2 H  ( x',\xi' ) + ( 1  - \mathcal O ( \| G \|_{C^2 }  )
( ( \xi - \xi')^2  + 
( x - x')^2 ) .
\end{split}
\]

We now move to construct the orthogonal projector 
\begin{equation}
\label{eq:projLa} \Pi_\Lambda ( L^2_\Lambda) = T_\Lambda ( L^2 ( \RR^n ) ) , \ \ 
\Pi_\Lambda^{*, H} = \Pi_\Lambda, \ \ \Pi_\Lambda^2 = \Pi_\Lambda , 
\end{equation}
and describe its structure. That is done similarly to the construction 
of $ \Pi_\varphi $ in \S \ref{s:we}. The complication comes from a more
involved form of the operators $ \zeta_j $ which requires the use of the almost analytic methods reviewed in \S \ref{s:anex}.

We start by defining operators which annihilate the deformed FBI transform. We first recall that the holomorphic extension of $ T$ 
satisfies
\[ Z_j  T \equiv 0, \ \ 
Z_j = h D_{z_j} - \zeta_j - 
 i h D_{\zeta_j } . \]
Hence, 
\begin{equation}
\label{eq:zetajL}
\begin{gathered}
 Z_j^\Lambda ( x, \xi , h D_x, h D_\xi ) T_\Lambda \equiv 0, \ \ 
  Z_j^\Lambda ( x, \xi, h D_x, h D_\xi ) := \left( h D_{z_j} - \zeta_j - 
  i h D_{\zeta_j } \right)|_{ \Lambda } , 
\end{gathered} 
  \end{equation}
where
\[ \begin{bmatrix} hD_{z }|_\Lambda \\ h D_{\zeta}|_\Lambda 
\end{bmatrix} 
= \begin{bmatrix} I + i  G_{x \xi} ( x, \xi ) & i G_{ \xi\xi} ( x, \xi) 
\\
- i G_{xx} ( x, \xi ) & I - i G_{ \xi, x } ( x, \xi ) \end{bmatrix}^{-1}
\begin{bmatrix} h D_{x} \\ h D_{\xi} \end{bmatrix} . \]
 Since $ Z_j $'s commute we have (with $ \alpha =  ( x, \xi ) $)
$ [ Z_j^\Lambda ( \alpha, h D_\alpha ) , Z_k^\Lambda ( \alpha, h D_\alpha ) ] = 0 $.

We now repeat the construction outlined in \cite{Sj96} and presented in the slightly simpler setting in \S \ref{s:we}. 
Again, the argument proceeds in the following
steps:

\begin{itemize}

\item construction of a uniformly bounded operator (as $ h \to 0 $) $ B_\Lambda : L^2_\Lambda \to L^2_\Lambda $ such that $ Z_j^{\Lambda} B_\Lambda = \mathcal O ( h^\infty)_{L^2_\Lambda \to L^2_\Lambda }  $,
$ B_\Lambda^{*, H} = B_\Lambda $ and
$ B^2_\Lambda = B_\Lambda + \mathcal O ( h^\infty )_{ L^2_\Lambda \to L^2_\Lambda}$;

\item characterization of the {\em unique} properties of the Schwartz kernel of 
$ B_\Lambda $: uniqueness of the phase and the determination of the amplitude from its restriction to the diagonal;

\item finding a projector $ P_\lambda = \mathcal O ( 1)_{ L^2_\Lambda \to 
L^2_\Lambda } $ onto the image of $ T_\Lambda $. 

\item choosing $ f \in S ( 1 )  $, $ f \geq 1/C $ so that 
$ A := P_\Lambda M_f P_\Lambda^{*, H}  $ (in the notation of \S \ref{s:FBI}), satisfies 
$ A = B_\Lambda + \mathcal O ( h^\infty )_{ L^2_\Lambda\to L^2_\Lambda }$; this relies on the uniqueness properties in the construction of $ B_\Lambda $;

\item expressing $ \Pi_\Lambda $ as a suitable contour integral of the resolvent of $ A $ 
and using it to show that $ \Pi_\Lambda  = B_\Lambda + \mathcal O ( h^\infty )_{ L^2_\Lambda \to L^2_\Lambda }$.
\end{itemize}

To construct $ B_\Lambda $ we postulate an ansatz
\begin{equation}
\label{eq:kerBLa}
\begin{gathered}  B_\Lambda u ( \alpha ) = h^{-n} \int_{T^* \TT^n}  
e^{ i  \psi ( \alpha, \beta )/h -   2 H ( \beta )/h } a ( \alpha, \beta ) u ( \beta ) d m_\Lambda ( \beta)  ,  \\ 
d m_\Lambda ( \beta ) := (\sigma|_\Lambda)^n / n! = 
d \alpha ,  \ \beta = \Re \alpha , \ \ \alpha \in \Lambda ,  
\end{gathered}
\end{equation}
and as in \eqref{eq:Kab0} 
\begin{equation}
\label{eq:Kab01} \begin{split}  Z_j^\Lambda  ( \alpha, h D_\alpha ) \left( e^{ i\psi ( \alpha, \beta )/h }
a ( \alpha, \beta ) \right) & = \mathcal O ( h^\infty + | \alpha - \beta |^\infty ) , \\
\widetilde Z_j^{\Lambda}  ( \beta, h D_\beta ) \left( e^{ i \psi ( \alpha, \beta )/h }
a ( \alpha, \beta ) \right) & = \mathcal O ( h^\infty + | \alpha - \beta |^\infty ) , 
\end{split} \end{equation}
where $ \widetilde Z_j^\Lambda $, $ j = 1, \cdots , n $  are defined
in \eqref{eq:defwideZ} below. The equations \eqref{eq:Kab01} are consistent with 
$  \psi ( \alpha, \beta ) = - \overline{ \psi ( \beta, \alpha ) } $and
$ a ( \alpha, \beta ) = \overline{ a ( \beta, \alpha ) }$ -- see
\S \ref{s:Laeik}.

\noindent
{\bf Notation.} Suppose $ Q $ is a differential operator with holomorphic 
coefficients defined near $ \Lambda $. We write 
\[  Q ( \alpha, h  D_\alpha ) = Q ( z , \zeta, h D_z , h D_\zeta ) ,\]
and $ \overline{ Q ( \alpha, h D_\alpha ) } $ for the corresponding
anti-holomorphic operator. The operator $ Q $ can be restricted to the 
totally real submanifold $ \Lambda $ and that restriction is
denoted by $ Q^\Lambda $. If we parametrize $ \Lambda $ by 
$ \alpha \in T^* \RR^n$ we write
$ Q^\Lambda = Q^\Lambda ( \alpha, h D_\alpha ) $. This operator then has an almost analytic extension to a neighbourhood of $ \Lambda $ and we
denote it by the same letter. We also consider the anti-holomorphic operator
$ u \mapsto \overline { Q^t \bar u } $,
\[   \int_\Lambda u ( \alpha ) [Q v] ( \alpha ) d\alpha = 
\int_\Lambda [ Q^t u ] ( \alpha ) v ( \alpha ) d \alpha ,\]
and denote its restriction of $ \Lambda $ by 
$ \bar Q^\Lambda $. The reason for this notation is the
fact that, as function on $ \Lambda \simeq T^* \RR^n $, 
\begin{equation}
\label{eq:sigLa}  \sigma ( \overline Q^\Lambda ) = 
\overline { \sigma ( Q^\Lambda ) } . 
\end{equation}
We use the same letter to denote its almost 
analytic extension to a neighbourhood of $ \Lambda $. 
We also define  $ \widetilde Q^\Lambda $, 
$ \sigma ( \widetilde Q^\Lambda ) ( \alpha, \alpha^*) = 
\overline { \sigma ( Q^\Lambda ) } ( \alpha , - \alpha^*) $.
Here $ \sigma $ refers to the semiclassical principal symbol.

We illustrate this in a simple example: $ \Lambda = \{ ( x, 
\xi - i g'(x )) : ( x, \xi ) \in \RR^2 \}$, $ g \in C^\infty 
( \RR; \RR ) $. If
$ Q = h D_z - \zeta - i hD_\zeta $ then 
\begin{gather*}
 Q^\Lambda = h D_x + i g''( x ) h D_\xi - \xi + i g'( x) - i h D_\xi , \\
 \bar Q^\Lambda = h D_x - i g'' ( x ) h D_\xi - \xi - i g'( x ) + 
 i h D_\xi ,\\
\widetilde Q^\Lambda =  - h D_x + i g''( x ) h D_\xi - \xi - {i} g'( x ) - i
h D_\xi ,
\end{gather*}
with the operators extended to a neighbourhood of $ \Lambda $ by 
taking holomorphic derivatives and an almost analytic extension of $ g $.
\qed

We note again that the weight $ H $ does not appear in 
$ \widetilde Z_j^\Lambda $. To see this
we first compute $ (Z_j^\Lambda) ^{* , H} $:
\[ \begin{split} \langle Z_j^\Lambda u , v \rangle_{L^2_\Lambda} & =
\int_{T^* \TT^n} Z_j^\Lambda ( \alpha, h D_\alpha ) u ( \alpha) 
\overline {v ( \alpha ) } e^{ - 2 H ( \alpha ) } d m_\Lambda ( \alpha ) 
\\
& = \int_{ \Lambda } Z_j ( \alpha , h D_\alpha ) 
u (\alpha ) \overline{ v ( \alpha ) } e^{ - 2 H ( \alpha ) / h }d \alpha \\
& = \int_{\Lambda } 
u ( \alpha ) ( (Z_j ( \alpha, h D_\alpha ))^t ( \overline{
v( \alpha ) } e^{ -2 H ( \alpha ) } ) d \alpha \\
& = \int_{\Lambda } 
u ( \alpha ) \overline {\left( e^{ 2 H( \alpha)/h } \overline Z_j ( \alpha, h D_\alpha )  e^{ -2 H ( \alpha )} \right)  v( \alpha ) } e^{ -2 H ( \alpha ) } ) d \alpha 
\end{split} \]
where (see the remark about notation above)
$ \overline Z_j ( \alpha, h D_\alpha ) v ( \alpha ) := \overline{ (Z_j ( \alpha, h D_\alpha )^t \overline{v ( \alpha) }} $.
Hence 
\[  \left(Z_j^{\Lambda }  ( \alpha, h D_\alpha) \right)^{*,H} 
=  e^{ 2 H( \alpha)/h } \overline Z_j^{\Lambda}  ( \alpha, h D_\alpha )  e^{ -2 H ( \alpha )} , 
\ \  \overline Z_j^{\Lambda}  = \overline Z_j |_{ \Lambda } . 
\]
We then have
 \begin{equation*}
\begin{split}   0 & \equiv ( Z_j^\Lambda B_\Lambda )^* u ( \alpha ) = B_\Lambda^* (Z_j^{\Lambda} )^* u ( \alpha ) = B_\Lambda (Z_j^{\Lambda} )^* u ( \alpha ) \\ & 
 = \int_{T^* \TT^n}  K_\Lambda ( \alpha, \beta ) e^{ -  2H ( \beta)/h }(Z_j^{\Lambda} )^* u ( \beta )
d m_\Lambda ( \beta)  \\
& =  \int_\Lambda K_\Lambda ( \alpha, \beta ) \overline Z_j ( \beta, h D_\beta ) \left( 
e^{ -  2H( \beta)/ h } u ( \beta ) \right) d  \beta  \\
& = 
\int \left ( \widetilde Z_j ( \beta,  h D_\beta ) K_\Lambda ( \alpha, \beta ) \right) 
u ( \beta ) e^{ - 2H( \beta ) /h } d \beta  \\ 
& = 
\int \widetilde Z_j^\Lambda ( \beta, h D_\beta ) K_\Lambda ( \alpha, \beta ) e^{ -2 H ( \beta ) /h } {u(\beta)} d m_\Lambda ( \beta) , 
\end{split}
\end{equation*}
where 
\begin{equation}
\label{eq:defwideZ}   \widetilde Z_j ( \beta, h D_\beta ) v ( \beta ) := 
\overline Z_j ( \beta , h D_\beta )^t   v( \beta ) = 
\overline { Z_j ( \beta , h D_\beta ) \overline{ v ( \beta ) } } ,
\ \ \ \widetilde Z_j^\Lambda := \widetilde Z_j |_{\Lambda } . 
\end{equation}
Explicitly we have
\begin{equation}
\label{eq:expzetL}
\begin{gathered}
\widetilde Z_j ( z, \zeta, h D_x, h D_\zeta ) = 
- h \bar D_{z_j} - \bar \zeta_j  
- i h \bar D_{\zeta_j } , \end{gathered} 
  \end{equation}
where
\[ \begin{bmatrix} h \bar D_{z }|_\Lambda \\ h \bar D_{\zeta}|_\Lambda 
\end{bmatrix} 
= \begin{bmatrix} I - i  G_{x \xi} ( x, \xi ) & -i G_{ \xi\xi} ( x, \xi) 
\\
 i G_{xx} ( x, \xi ) & I + i G_{ \xi, x } ( x, \xi ) \end{bmatrix}^{-1}
\begin{bmatrix} h D_{x} \\ h D_{\xi} \end{bmatrix} . \]
 Also,
$  [ \widetilde Z_j^\Lambda, \widetilde Z_k^\Lambda ] = 0 .$

\subsection{Eikonal equations}
\label{s:Laeik}

Let $ \zeta_j^\Lambda  $ and $ \widetilde \zeta_j^\Lambda $ be 
the principal symbols of $ Z_j^\Lambda $ and $ \widetilde Z_j^\Lambda$ respectively.  {The} eikonal equations we want to solve are
\begin{equation}
\label{eq:Laeik}
\zeta_j^\Lambda ( \alpha, d_\alpha \psi ( \alpha, \beta ) ) = 
\mathcal O ( | \alpha - \beta |^\infty ) , \ \ 
\widetilde \zeta_j^\Lambda ( \beta, d_\beta \psi ( \alpha, \beta ) ) 
= \mathcal O ( | \alpha - \beta |^\infty ) ,  \ \ \alpha, \beta \in \Lambda. 
\end{equation}
We recall that $ \zeta_j^\Lambda $ are restrictions to $ T^* \Lambda $ 
of holomorphic functions on $ T^* \CC^{2n} $: $ \zeta_j = 
x_j^* - \xi_j - i \xi_j^* $, $ ( x, \xi, x^* , \xi^* ) \in 
\CC^{2n} \times \CC^{2n} $. We now put 
\[ \overline{ \zeta}_j^\Lambda ( \alpha, \alpha^* ) := 
\widetilde {\zeta}_j^\Lambda ( \alpha, - \alpha^* ) \]
which is the principal symbol of $ \overline Z_j^\Lambda  $. 

From the geometric point of view, so that we remain in the same framework as in \S \ref{s:geop}, it is convenient to construct the phase 
function corresponding to 
$ B_H := e^{ - H/h } B_\Lambda e^{H /h} $. That means that
properties of $ B_\Lambda $ on $ L^2 (\Lambda) $ are equivalent to the properties of $ B_H $ on $ L^2 $, that is 
we want
\begin{equation}
\label{eq:propBH}  B_H = B^*_H , \ \  B^2_H = B_H .
\end{equation}
We have 
\[ \begin{gathered} B_H u ( \alpha) = h^{-n} \int e^{ \frac i h \psi_H ( \alpha, \beta ) }
a ( \alpha, \beta ) u ( \beta ) d \beta , \\ 
\psi_H ( \alpha, \beta ) : = i H ( \alpha ) + \psi( \alpha, \beta ) +  i H ( \beta ) .
\end{gathered} \]

To simplify the notation we first assume that $ \Lambda $ (and consequently
$ H $ defined by \eqref{eq:defH} and $ \zeta_j^H $, $ \bar \zeta_j^H $) are analytic. We will replace that by almost analyticity by proceeding
as in \S \ref{s:anex}. 

To construct $ \psi_H $ we consider $ \mathscr C_H $,  the relation associated to it: 
\begin{equation}
\label{eq:defCH}  \mathscr C_H = \{ ( \alpha, d_\alpha \psi_H ( \alpha, \beta ) , 
\beta, - d_\beta \psi_H ( \alpha, \beta ) ) : ( \alpha, \beta)  \in 
\nbhd_{\CC^{4n} } ( \Diag ( \Lambda \times \Lambda ) ) \}. \end{equation}
In view of \eqref{eq:propBH} we must have 
\begin{equation}
\label{eq:propCH}   \mathscr C_H \circ \mathscr C_H = \mathscr C_H , \ \ 
\overline{ \mathscr C}_H^t  = \mathscr C_H .\end{equation}
(Here $ \overline{ \mathscr C}_H^t := 
\{ ( \bar \rho, \bar \rho' ) : ( \rho' , \rho ) \in \mathscr C_H \} $, 
and $ \rho \mapsto \bar \rho $ is defined after the almost analytic identification of
$ \Lambda $ with $ T^* \RR^n $.)

We define
\begin{equation}
\label{eq:defzetH}   \begin{split} \zeta_j^H ( \alpha, \alpha^* ) & :=
\zeta_j^\Lambda ( \alpha, \alpha^* - i dH( \alpha ) ), 
\\  \bar \zeta_j^H  ( \alpha , \alpha^* ) & := 
\bar  \zeta_j^\Lambda  ( \alpha, \alpha^* + i dH( \alpha ) )
= \overline{ \zeta_j^H ( \bar \alpha, \bar \alpha^* ) }
 , \end{split} \end{equation}
 so that the formal analogue of \eqref{eq:Laeik}
is given by 
\begin{equation}
\label{eq:eikH} \begin{gathered} \zeta_j^H ( \alpha, d_\alpha \psi_H ( \alpha, \beta ) )= 0
, \ \ \ \bar \zeta_j^H ( \alpha, - d_\beta \psi_H ( \alpha, \beta ) )= 0.
\end{gathered}
\end{equation}
(Here again the $ \bar \alpha $ and $ \bar \alpha^* $ are defined after an identification of $ \Lambda $ with $ T^* \RR^n $). We construct $ \mathscr C_H $ geometrically -- see \S \ref{s:geop} for the simpler linear algebraic treatment in the case of the FBI transform without weights. In view of \eqref{eq:defCH} and \eqref{eq:eikH} 
we must have
\[ \begin{gathered}  \mathscr C_H \subset S \times \overline S , \ \ 
S := \{ \rho : \zeta_j^H ( \rho ) = 0 , \ \rho \in \nbhd_{ \CC^{4n}} 
( T^* \Lambda ) \} , \\ \overline S := \{ \bar \rho : \rho \in S \}=
\{ \rho : \bar \zeta_j^H ( \rho ) = 0 , \ \rho \in \nbhd_{ \CC^{4n}} 
( T^* \Lambda ) \} . \end{gathered} \]
If follows that the complex vector fields $ H_{\pi^*_L \zeta_j^H} $ and $ H_{\pi_R \bar \zeta_j^H}$
($ \pi_L ( \rho, \rho' ) := \rho$, $ \pi_R ( \rho, \rho' ) := \rho' ) $)
are tangent to $ \mathscr C_H$. By checking the case of $ T^* \Lambda = 
T^* \RR^n $ (no deformation and hence $ H \equiv 0 $) we see, as in 
\S \ref{s:geop} that $ S \cap \bar S $ is a symplectic submanifold 
(with respect to the complex symplectic form) of complex dimension 
$ 4 n $. The independence of $ H_{\zeta_k^H } $, $ H_{\bar \zeta_j^H } $
, $ j , k = 1, \cdots n $ (again easily seen in the unperturbed case)
shows that 
\[  B_{\CC^n} ( 0 , \epsilon ) \times B_{\CC^n} ( 0 , \epsilon ) 
\times ( S \cap \bar S ) \ni ( t, s, \rho)  \mapsto 
  ( \exp \langle t, H_{\zeta^H} \rangle  ( \rho ) , 
\exp \langle s , H_{ \bar \zeta_H} \rangle  ( \rho ) ) \in 
\CC^{8n} , \]
is a bi-holomorphic map to an embedded (complex) $ 4n$ dimensional
submanifold. 
This and idempotence (first condition in \eqref{eq:propCH}) imply that
\[ \mathscr C_H =  \left\{ (
\exp \langle t , H_{\zeta^H} \rangle ( \rho) , 
\exp \langle s , H_{\bar \zeta^H} \rangle ( \rho )) : \rho \in 
S \cap \bar S , \ \ t , s \in B_{ \CC^n } ( 0 , \epsilon ) \right\}  , \]
where $ \langle t, H_{\bullet^H } \rangle := 
\sum_{ k=1}^n t_k H_{\bullet_k^H } $, $ \bullet = 
\zeta, \bar \zeta $.

The second condition in \eqref{eq:propCH} is automatically satisfied
(this makes sense since $ \mathscr C_H \subset S \times \bar S $
came from demanding that $ 0 = (\zeta_j^H B )^* = B (\zeta_j^H)^* $). 
Since $ \pi : \mathscr C_H \to \nbhd_{ \CC^{4n} } ( \Lambda \times \Lambda ) $ is surjective we have have a parametrization 
given by \eqref{eq:defCH} with $ \psi_H $ determined up to an additive constant. We claim that we can choose that constant so that
\begin{equation}
\label{eq:psiH0}  \psi_H ( \alpha, \alpha ) = 0 . 
\end{equation} 
To see this we note that (from $ \mathscr C_H = \overline {\mathscr C}_H ^t $)
\[ d_\alpha \psi_H  ( \alpha , \beta ) |_{\alpha = \beta } = - \overline { d_\beta \psi_H  ( \alpha, \beta )} |_{\alpha= \beta} ,  \ \alpha 
\in \Lambda ,  \]
and hence 
\begin{equation}
\label{eq:dalphpsiH} \begin{split}  
d_\alpha ( \psi_ H(\alpha, \alpha ) ) & = 
d_\alpha \psi_H ( \alpha, \beta ) |_{\alpha = \beta } + 
d_\beta  \psi_H ( \alpha, \beta ) |_{\alpha = \beta } \\
& = 
2 i \Im d_\alpha \psi_H ( \alpha, \beta ) |_{ \alpha=  \beta } , \ \
\alpha \in \Lambda .
\end{split} \end{equation}

To find $ \Im d_\alpha \psi_H ( \alpha, \beta ) |_{ \alpha=  \beta } $
it is convenient to go to the origins of the symbols
$ \zeta_j^H $ \eqref{eq:defzetH} : $ Z_j^\Lambda $'s, with 
symbols $ \zeta_j^\Lambda $ annihilate the phase in 
$ T_\Lambda $ 
and hence
\begin{gather*} S_\alpha := S \cap T_\alpha^* \Lambda^\CC = 
\left\{ ( \alpha, d_\alpha \varphi ( \alpha, y ) 
+ i dH ( \alpha ) ) \
 :  \ y \in \CC^n \right\}, \\ 
\varphi ( \alpha, y ) := \langle z - y , \zeta \rangle + 
i ( z - y )^2 /2  , \\
z = \alpha_x + i G_\xi ( \alpha_x, \alpha_\xi ) , \ \
\zeta = \alpha_\xi - G_x ( \alpha_x , \alpha_\xi ) .
\end{gather*}
In the case $ G = 0 $  (and hence $ H = 0 $), 
$ S_\alpha  $ and $ \bar S_\alpha := \bar S \cap T^*_\alpha
\Lambda ^\CC  $ intersect transversally 
in one point and that has to remain true under perturbations. Hence we 
are looking for a solution to 
\begin{equation}
\label{eq:dalH}  d_\alpha \varphi ( \alpha, y ) + i d H ( \alpha ) = 
\overline{ d_\alpha \varphi ( \alpha, y' ) } - i d H ( \alpha ) .\end{equation} 
Now, at $ y = y' = \alpha_x $ we have 
$ d_\alpha \varphi  ( \alpha, y ) = \zeta d z|_\Lambda $
and in view of the definition of $ d H $ in \eqref{eq:defH}, 
\eqref{eq:dalH} holds. It follows that  {for $\alpha\in \Lambda$}
\[  S_\alpha \cap \bar S_\alpha = \{ ( \alpha, \Re (\zeta d z   |_\Lambda ) \} \in T^* \Lambda . \]
{Next, by analytic continuation (replaced by almost analytic continuation below),} it follows (since intersection of $ S $ and $ \bar S $ is transversal 
and we have the right dimension) that
\begin{equation}
\label{eq:ScapbarS} \mathscr J := S \cap \bar S = \{ ( \alpha  , \omega ( \alpha ) + \bar\omega (  \alpha )   ) : \alpha \in \nbhd_{ \CC^{2n}} ( \Lambda ) \} , \ \ 
\omega ( \alpha )|_\Lambda  = 
 \tfrac12 \Re ( \zeta d z |_\Lambda) , \end{equation}
  {where we recall that $\bar{\omega}(\alpha)=\overline{\omega(\bar{\alpha})}$.}
 But this shows that $ \pi^{-1} ( \diag ( \Lambda \times \Lambda ) )
\cap \mathscr C_H $ is real which means that $ \Im d_\alpha \psi_H ( \alpha , \beta ) {\big|_{\beta=\alpha}} = 0 $ for $ \alpha \in \Lambda $ showing 
that $ \psi_H ( \alpha , \alpha ) $ is  {a} constant which can be chosen be $ 0 $. This gives \eqref{eq:psiH0}.

\noindent
{\bf Remark.} Vanishing of $\Im d_\alpha \psi_H ( \alpha , \beta ) $
also shows that 
\[ - \Im \psi_H  ( \alpha, \beta ) = 
\mathcal O ( | \alpha - \beta|^2 ) \]
and since $ G $ is assumed to be small, the case of $ G= 0 $ shows that
\begin{equation}
\label{eq:HpsiH}
- \Im \psi_H ( \alpha, \beta ) 
\leq - | \alpha - \beta|^2 / C , \ \ C> 0 .
\end{equation}
This shows that $ B_H $ given by \eqref{eq:kerBLa} is bounded on 
$ L^2 $. 
\qed

We now comment on the general case and explain how to use 
almost analytic extensions off $ \Lambda $. We first identify 
$ \Lambda $ with $ T^* \RR^n $ using \eqref{eq:LaG} and extending 
$ G $ almost analytically to $ \CC^{4n} $.
We then define $ \mathscr J $ by \eqref{eq:ScapbarS} using an 
almost analytic extension of $ \omega ( \alpha)|_\Lambda $ (where
$ \bar\omega ( \alpha ) = \overline {\omega ( \bar \alpha ) } $).
We are now basically in the same situation as in \S \ref{ss:eik}, except for a larger number of vector fields, with $ \Lambda $ replaced by  $ \mathscr C_H $ and $ \Lambda_0 $  by $ \{ ( \rho, \rho ) : \rho \in \mathscr J \} $. 
Hence we define
\[ \mathscr C_H = 
\left\{ \left( \exp{ \widehat{ \langle t, H_{\zeta^H } \rangle }} (\rho), 
\exp{ \widehat {\langle s, H_{\bar \zeta^H } \rangle } } (\rho) \right) :
\rho \in \mathscr J , \ \ t, s \in B_{\CC^n } ( 0 , \epsilon ) \right\} . \]
 {Almost the} same arguments as in \S \ref{ss:eik}
show that 
{ \begin{equation}
\label{eq:sJt}  |\Im \exp{ \widehat{ \langle t, H_{\zeta^H } \rangle }} (\rho) | \geq |t|/ C , \ \ 
| \Im \exp{ \widehat {\langle s, H_{\bar \zeta^H } \rangle } } (\rho) |
\geq |s|/ C ,  \ \ \rho \in \mathscr J . \end{equation}
In fact, for $ \zeta_j := z_j^* - \zeta_j - i \zeta_j^* $ and $ 
\bar \zeta_j := z_j^* - \zeta_j + i \zeta_j^* $ we have $ \{ \zeta_j , \bar \zeta_k\} = 2 i \delta_{jk} $. 
 On $ T^* \Lambda $, 
$  \bar \zeta_k^H = \overline{\zeta_k^H} $ 
and $ \{ \zeta_j^H, \bar \zeta_k^H \}/2i   $ is positive definite. By
taking a linear combination of $ \zeta_j^H $'s we can then arrange that,
at a given point, $  \{ \zeta_j^H, \bar \zeta_k^H \}/2i  = \delta_{jk} $.
We can then make a linear
symplectic change of variables at any point of $ T^* \Lambda $ giving
new variables $  ( x, y, \xi,\eta )$, $ x, y , \xi , \eta \in \RR^n $, centered at $ 0 \in \RR^{4n} $,  
such that 
\[ \zeta_j^H = c ( \eta_j + i y_j ) + \mathcal O ( |x|^2 + |y|^2 + |\xi|^2 + |\eta|^2 )  , \ \ c >0 , \]
and this holds also for almost continuations of $ \zeta_j^H $. 
That means that near $ 0 $,  
\begin{equation}
\label{eq:mathsJ}  \mathscr J = \{ ( z, 0 , \zeta, 
0 ) + F ( z, \zeta ) ) : ( z, \zeta ) \in \nbhd_{ \CC^{2n} } ( 0 )  \} , \ \ F = \mathcal O ( |z|^2 + |\zeta|^2 ) ,
\end{equation}
We also note that for $ ( z, \zeta )\in \RR^{2n} $ (which corresponds to 
the interection with $ T^* \Lambda $), $ \mathscr J $ is real. This means
that in \eqref{eq:mathsJ}, 
\[ \Im F ( z , \zeta )  = 
\mathcal O ( ( |\Im z | + |\Im \zeta | ) ( |z| + |\zeta| ) ). \]
Hence,
\[ \begin{split} 
|\Im \exp{ \widehat{ \langle t, H_{\zeta^H } \rangle }} ( 
( z, 0 , \zeta, 
0 ) + F ( z, \zeta ) ) | & =
|( \Im z , c \Im t , \Im \zeta, c \Re t ) | \\
& \ \ \ \ + \mathcal O (  ( |\Im z | + |\Im \zeta | ) ( |z| + |\zeta| )  + |t|^2 ) 
\\ & \geq |t| /C , \ \text{ if $ |z|,|\zeta| \ll 1 $,}  \end{split}
\]  
with the corresponding estimate for $ \bar \zeta^H $. 
Lemma \ref{l:flow1} and \eqref{eq:sJt} now show the almost analyticity of $ \mathscr C_H $.} 
As in the proof of \eqref{eq:paraLa} we now obtain $ \psi_H = 
\psi_H ( \alpha, \beta ) $ such that, 
\[ d_{\bar \alpha, \bar \beta } \psi_H  ( \alpha, \beta ) = 
\mathcal O \left( d ( ( \alpha, \beta ) , \diag (\Lambda \times \Lambda ) )^\infty  
\right) . 
 \]
Restricting $ \psi_H ( \alpha, \beta ) $ to $ \Lambda \times \Lambda $ 
gives \eqref{eq:Laeik}. 


We now return to our original $ \psi $ in \eqref{eq:kerBLa}, 
$ \psi( \alpha, \beta ) = - i H ( \alpha ) + \psi_{H } ( \alpha, \beta) - i H( \beta  ) $. Our construction shows that 
\begin{equation}
\label{eq:psi2H}
\text{\eqref{eq:Laeik} holds, }   \ \psi ( \alpha, \alpha ) = - 2 i H ( \alpha ) , \  \ \psi ( \alpha, \beta ) = - \overline{ \psi ( \beta, \alpha ) },  \ \ 
\alpha, \beta  \in \Lambda . 
\end{equation}

\noindent
{\bf Remark.} Although we motivated our construction using 
the self-adjointness and idempotence properties of the operator $ B_\Lambda $ (or equivalently $ B_H $), the construction shows
that 
 $ \psi $ is uniquely determined, up to $ \mathcal O ( | \alpha - \beta|^\infty ) $, by \eqref{eq:psi2H}. \qed

We also record that
\begin{equation}
\label{eq:proppsi}
\begin{gathered}
 {\rm{c.v.}}_\beta \left( \psi ( \alpha, \beta ) + 2  i H( \beta ) + 
\psi ( \beta, \alpha ) \right) = \psi ( \alpha, \alpha ) , \\
- H ( \alpha ) - \Im \psi ( \alpha, \beta ) - H( \beta )
\leq - | \alpha - \beta|^2 / C , \ \ C> 0 .
\end{gathered}
\end{equation}

\subsection{Transport equations}
\label{s:traLa}

We now return to \eqref{eq:Kab01} and consider the transport equations
satisfied by $ a $. The analysis is similar to that in 
\S \ref{ss:te} and we start with a formal discussion (valid when 
all the objects are analytic). In view of the eikonal equations we have, 
as in \eqref{eq:traspa}, 
\[ a ( \alpha, \beta ) \sim \sum_{k=0}^\infty h^k a_k ( \alpha, \beta ) , \]
where, with $ \zeta_{j1} ^\Lambda $ the second term in the expansion of
the symbol of $ Z_j^\Lambda $, we want to solve
\begin{equation}
\label{eq:ZjLa0} \begin{gathered} 
V_j   a_k (\alpha, \beta) + c_j ( \alpha, \beta ) 
a_k ( \alpha, \beta)  = F^j_{k-1} ( a_0, \cdots, a_{k-1} ) ( \alpha, \beta)  , \ \ F^j_{-1} \equiv 0, \\
V_j := \langle V_j ( \alpha, \beta ) , \partial_\alpha \rangle, \ \ V_j (\alpha, \beta )_\ell  :=  \partial_{ \alpha_\ell ^*}  \zeta_j^\Lambda ( \alpha, d_\alpha \psi ( \alpha, \beta ) )    , \\ 
c_j ( \alpha , \beta  ) := \tfrac12 \sum_{\ell=1}^{2n}  
\partial_{\alpha_\ell}  V_j (\alpha, \beta ) + 
\zeta_{j1} ( \alpha , d_\alpha \psi ( \alpha, \beta ) {)}  \\ 
\ \ \ \ \ \ \ \ \ \ \  \ \ \ \ \ \ \ \ \ \ \ \ \ \ \ \ \ \ \ 
- \, i \sum_{k,\ell = 1}^{2n }
\partial_{\alpha_k \alpha_\ell}^2 \psi ( \alpha, \beta  ) \partial_{\alpha^*_k \alpha^*_\ell}^2 \zeta_j^\Lambda  ( \alpha , d_\alpha \psi ( \alpha, \beta  ) ) .
\end{gathered} \end{equation}
We have similar expressions coming from the applications $ \widetilde Z_j^\Lambda ( \beta , h D_\beta ) $ with 
$ V_j $, $ c_j $, $ F^j_k $ replaced by $ \widetilde V_j $, $ \widetilde 
c_j $, $ \widetilde F_k^j$, and  with the roles of $ \alpha $ and $ \beta $ switched. 
A key observation here is that $ H_{\zeta_j^\Lambda ( \alpha ) } $ and
$ H_{\bar \zeta_j^\Lambda ( \beta )  } $ are tangent to $ \mathscr C $  {and commute}
and that $ V_j $ and $ - \widetilde V_j $ are these vector fields in the parametrization of $ \mathscr C $ by $ ( \alpha, \beta ) $. Hence,
\begin{equation}
\label{eq:comVjk0}
[ V_j , V_k ] = 0 , \ \ [ V_j, \widetilde V_k ] = 0  , \ \
[ \widetilde V_k, \widetilde V_k ] = 0 . 
\end{equation}
We note that,  {for any $b(\alpha,\beta)\in S^0$},
\begin{equation}
\label{eq:ZjLa} 
\begin{split}  &  Z_j^\Lambda ( \alpha, h D_\alpha) \left( 
e^{ \frac i h \psi ( \alpha, \beta ) } b ( \alpha, \beta ) \right) 
= h e^{ \frac i h \psi ( \alpha, \beta ) }(   ( V_j   + c_j  ) b ( \alpha, \beta )  + \mathcal O ( h  ) ) , \\
&  \widetilde Z_j^\Lambda ( \beta, h D_\beta) \left( 
e^{ \frac i h \psi ( \alpha, \beta ) } b ( \alpha, \beta ) \right) 
= h e^{ \frac i h \psi ( \alpha, \beta ) } ( ( \widetilde V_j   + \widetilde c_j  ) b ( \alpha, \beta )   + \mathcal O ( h ) 
).
\end{split} 
\end{equation}
 {Moreover, solving~\eqref{eq:ZjLa0} means that}
\begin{equation}
\label{eq:ZjLa1} 
\begin{split} 
& Z_j^\Lambda ( \alpha, hD_\alpha ) \left( e^{ \frac i h \psi ( \alpha, \beta ) } \sum_{k=0}^{K-1} h^k a_k ( \alpha, \beta ) \right) =
h^{ K+1 } e^{ \frac i h \psi ( \alpha, \beta ) } F_{ {K-1}}^j ( \alpha, \beta ) , \\
& \widetilde Z_j^\Lambda ( \beta, hD_\beta ) \left( e^{ \frac i h \psi ( \alpha, \beta ) } \sum_{k=0}^{K-1} h^k a_k ( \alpha, \beta ) \right) =
h^{ K+1 } e^{ \frac i h \psi ( \alpha, \beta ) } \widetilde F_{ {K-1}}^j ( \alpha, \beta ) . \end{split}
\end{equation}
Since 
\begin{gather*}  [ Z_j^\Lambda ( \alpha, h D_\alpha ) , Z_k^\Lambda ( \alpha, h D_\alpha )] = 0 , \ \ 
[\widetilde Z_j^\Lambda ( 
\beta, h D_\beta ) , \widetilde Z_k^\Lambda ( 
\beta, h D_\beta ) ] = 0 , \\
[ Z_j^\Lambda ( \alpha, h D_\alpha ), \widetilde Z_k^\Lambda ( 
\beta, h D_\beta ) ] = 0 , \end{gather*}
we have from \eqref{eq:comVjk0} and \eqref{eq:ZjLa},
\begin{equation}
\label{eq:comVjk}  V_j c_k = V_k c_j , \ \ 
V_k \widetilde c_j = 
\widetilde V_j c_k, \ \
\widetilde V_k \widetilde c_j = \widetilde V_j \widetilde c_k . 
\end{equation}
Similarly, \eqref{eq:ZjLa1} gives 
\begin{equation}
\label{eq:comVjk1} 
\begin{gathered}  ( V_j + c_j ) F_{ {K-1}}^\ell = ( V_k + c_k ) F_{ {K-1}}^j , \ \
( \widetilde V_j + \widetilde c_j ) \widetilde F_{ {K-1}}^\ell = ( \widetilde V_k + \widetilde c_k ) \widetilde F^j_{ {K-1}} , \\
( V_j +  c_j ) \widetilde F_{ {K-1}}^\ell = ( \widetilde V_k + \widetilde c_k )  F^j_{ {K-1}} . 
\end{gathered} \end{equation}
Equations \eqref{eq:comVjk} and \eqref{eq:comVjk1} provide compatibility 
conditions for solving \eqref{eq:ZjLa0}:
\[   ( V_j + c_j ) a_k = F_{k-1}^j ,  \ \ ( \widetilde V_\ell + 
\widetilde c_\ell ) a_k = F_{k-1}^\ell , \ \ 
a_k ( \alpha, \alpha ) = b_k ( \alpha ) , \]
where  {the} $ b_k $'s  {are} prescribed. In fact, since  {the} $ V_\ell $'s and 
$ \widetilde V_j $ {'s} are independent when $ \alpha = \beta $ (as complex vectorfields), 
\[ \begin{gathered}  \CC^{2n}  \times \CC^n \times \CC^n \ni
( \rho, t , s ) \mapsto ( \alpha , \beta ) = \left( \exp 
\langle V, t \rangle  ( \rho  ) , 
\exp  \langle \widetilde V, s \rangle  (\rho ) \right) \in 
\CC^{2n} \times \CC^{2n} ,\\
\langle V, t \rangle :=  \sum_{j=1}^n t_j V_j  , \ \
\langle \widetilde V, s \rangle  := \sum_{\ell=1}^n s_j \widetilde V_\ell , 
\end{gathered} \]
is a local bi-holomorphic map onto 
of $ \nbhd_{ \CC^{4n} } ( \diag ( \Lambda \times \Lambda) ) $ (almost analytic in the general case). In view of 
this and of \eqref{eq:comVjk0}, \eqref{eq:comVjk},  {the following integrating factor}, $ g = g ( \alpha, \beta ) $,  is well defined (in the analytic case) on $ \nbhd_{ \CC^{4n} } ( \diag ( \Lambda \times \Lambda) ) $:
\[  g ( e^{ \langle V, t \rangle  }  ( \rho) , 
e^{ \langle  \widetilde V, s \rangle }  ( \rho ) ) := - 
\sum_{ j=1}^n \int_0^1 ( 
t_j c_j + s_j \widetilde c_j )|_{( \alpha, \beta) = ( e^{ \tau \langle V, t \rangle  }  ( \rho ) , 
e^{ \tau \langle  \widetilde V, s \rangle }  ( \rho ) ) } d\tau , \]
 {and satisfies}
\[ V_j g ( \alpha , \beta ) = c_j ( \alpha, \beta ) , \ \ 
 \widetilde V_j g ( \alpha , \beta ) = \widetilde c_j ( \alpha, \beta ), \ \ j =1, \cdots , n . \] 
We then define $ a_k ( \alpha, \beta ) $ inductively as follows: at
$ ( \alpha, \beta ) = 
( e^{  \langle V, t \rangle  }  ( \rho ) , 
e^{  \langle  \widetilde V, s \rangle }  ( \rho ) ) $, 
\[ \begin{split}   a_k ( \alpha, \beta ) & =  
e^{g(\alpha, \beta  ) } 
 b_k ( \rho ) \\
& + e^{g (\alpha, \beta  ) } \int_0^1 e^{-g ( \gamma, \gamma' ) } (t_j F_{k-1}^j 
( \gamma, \gamma' )  + s_j \widetilde F_{k-1}^j ( \gamma, \gamma' ) )|_{( \gamma, \gamma') = ( e^{ \tau \langle V, t \rangle  }  ( \rho ) , 
e^{ \tau \langle  \widetilde V, s \rangle }  ( \rho ) ) } d\tau .
\end{split} \]
The compatibility relations \eqref{eq:comVjk1} then show that \eqref{eq:ZjLa0} hold.

We now modify this discussion to the $ C^\infty $ case using 
almost analytic extensions as in \S \ref{ss:te} and that provides
solutions of \eqref{eq:ZjLa0} for $ ( \alpha, \beta ) 
\in \Lambda \times \Lambda $ valid to infinite order at 
$ \diag( \Lambda \times \Lambda ) $ with any initial data on the diagonal. In the time honoured tradition of \cite{HS} and \cite{Sj96} we omit the tedious details. 

Combination of \S \ref{s:Laeik} and 
\ref{s:traLa} gives \eqref{eq:Kab0} with arbitrary $ a ( \alpha, \alpha ) \sim \sum_{ k } b_k ( \alpha ) h^k $.

\subsection{Construction of the projector}
\label{s:coaLA}

We now proceed as in \S \ref{s:ams} and obtain the initial values, $ b_k ( \alpha ) $ in the construction of the amplitude. Thus let $ B_\Lambda $ be given by \eqref{eq:kerBLa} with phase and amplitude
satisfying \eqref{eq:Kab01} and \eqref{eq:HpsiH} with $ a ( \alpha, \alpha) $ to be chosen. We also note that \eqref{eq:Kab01} 
determine $ a( \alpha, \beta ) $ (up to $ \mathcal O ( | \alpha - \beta|^\infty + h^\infty )$)
from $ a ( \alpha, \alpha ) \sim \sum_{ k } b_k ( \alpha ) h^k  $. 

To find $ b_k$'s we proceed by computing the expansion of 
$ B_\Lambda^2 $ using stationary phase. Since $ Z_j ^\Lambda B_\Lambda^2 
= \mathcal O ( h^\infty )_{L^2 \to L^2 } $ and $ B_\Lambda^2 $ is 
self-adjoint, the integration kernel of $ B_\Lambda^2 $ is again determined by its values on the diagonal (for the phase see the remark after \eqref{eq:psi2H}).  {The} stationary phase argument \eqref{eq:KB21}-\eqref{eq:a2b} gives the desired $ b_k's $ and we obtain 
$ a ( \alpha, \beta ) $ with 
\[   a ( \alpha , \beta ) \sim \sum_{k=0}^\infty h^k a_k ( \alpha, \beta ) , \ \ 
| a_0( \alpha, \alpha ) | > 1/C , \]
such that $ B_\Lambda $ given by 
\eqref{eq:Kab0} satisfies
\begin{equation}
\label{eq:Kab2}
B_\Lambda^{*,H} = B_\Lambda, \ \ B_\Lambda = \mathcal O ( 1 ) : L^2_\Lambda 
\to L^2_\Lambda , \ \ B_\Lambda^2 = B_\Lambda .
\end{equation} 

We now proceed as in \S \ref{s:cotp} and show that the exact
orthogonal projector \eqref{eq:projLa} satisfies
\begin{equation}
\label{eq:B2Pi1}
\Pi_\Lambda = B_\Lambda + \mathcal O ( h^\infty )_{L^2_\Lambda \to L^2_\Lambda } .
\end{equation}
The only difference in the argument is the construction of the  
the exact projector $ P_\Lambda $:
\[ P_\Lambda ( L^2_\Lambda ( \Lambda ) ) = T_\Lambda ( L^2 ( \RR^n )) , \ \ 
P_\Lambda^2 = P_\Lambda , \ \ P_\Lambda = \mathcal O ( 1 )_{L_H ( \Lambda ) \to L_H ( \Lambda ) } . \]
But a bounded projection was already provided by Lemma \ref{l:TSLa} and in its notation we can take 
\[ P_\Lambda = T_\Lambda S_\Lambda . \]
For $ f \in S (  \Lambda ) $, $ f ( \alpha ) \sim \sum_{ k=0}^\infty 
f_k ( \alpha ) h^k$ , $ f_0 ( \alpha) > 1/C  $, we now define
\[ A_f :=  P_\Lambda f P_\Lambda^{*, H } , \ \ 
A_f u ( \alpha ) =: h^{-n} \int_\Lambda e^{ \frac i h \psi_1 ( \alpha, \beta ) } a_f ( \alpha, \beta ) u ( \beta ) e^{ - 2 H ( \beta ) /h} dm_\Lambda ( \beta ) . \]
As in \S \ref{s:cotp} we claim that $\psi_1 = \psi $ (modulo 
$ \mathcal O ( | \alpha - \beta|^\infty ) $). 
Indeed, 
since $ A_f^{ *, H} = 
 {A_f} $ and $ P_\Lambda = T_\Lambda  S_\Lambda $, 
the arguments leading to \eqref{eq:Laeik} apply and those eikonal equations hold for $ \psi_1 $ as well. Hence, $ \psi_1 $ is fully 
determined by its value on the diagonal and we find 
that using $ \psi_\Lambda $ in \eqref{eq:psiLa} and 
\eqref{eq:dImpsi} 
\[ \begin{split} \psi_1 ( \alpha, \alpha )  + 2 i H ( \alpha ) & = 
{\rm{c.v.}}_{\beta } \left( \psi_\Lambda ( \alpha, \beta ) - 
\overline{ \psi_{\Lambda } ( \alpha, \beta ) } \right) = 0    . 
\end{split} \]
But this means that \eqref{eq:psi2H} holds for $ \psi_1 $ and hence
 $ \psi_1 ( \alpha, \beta ) = \psi( \alpha, \beta ) +  \mathcal O ( |\alpha - \beta |^\infty ) $. Choosing $ f $ so that 
 $ A_f = B_\Lambda + \mathcal O ( h^\infty)_{L^2_\Lambda \to L^2_\Lambda} $
as in \S \ref{s:cotp} and arguing as in that section gives 
\eqref{eq:B2Pi1}.

\section{Pseudodifferential operators}
\label{s:pseudoL} 
We now discuss the action of pseudodifferential operators
\begin{equation}
\label{eq:Pps} P u ( x ) = \frac{1}{(2 \pi i h)^n} \int_{\RR^n}\!\int_{\RR^n} 
p ( x, \xi ) e^{ \frac{i}{h} \langle x - y , \xi \rangle}  u ( y) dy d \xi, 
\end{equation}
where $ p $ has a holomorphic extension satisfying
\begin{equation} 
\label{eq:pps}
| p ( z, \zeta ) | \leq M , \ \ | \Im z |\leq a,  \ \ | \Im \zeta | \leq b , 
\end{equation}
for some $ a, b, M > 0$. 

We start with the following
\begin{lemm}
\label{l:TPSLa}
Suppose that $ P $ is given by \eqref{eq:Pps}. Then 
\begin{equation}
\label{eq:TPSLa}
\begin{gathered}
T_\Lambda P S_\Lambda = \mathcal O ( 1 ) : L^2_\Lambda \to L^2_\Lambda , 
\ \ \ 
T_\Lambda P S_\Lambda = c_0h^{-n} \int_\Lambda 
K_P ( \alpha, \beta ) u ( \beta) d \beta , \\
K_P ( \alpha, \beta ) = e^{ \frac i h \psi_\Lambda ( \alpha, \beta ) } 
a_P ( \alpha, \beta )+r(\alpha,\beta) , \\ 
a_P ( \alpha , \beta ) \sim \sum_{ j =0}^\infty h^j a_P^j ( \alpha, \beta ) , \ \ a_P^0 ( \alpha, \alpha ) = p|_{ \Lambda } ( \alpha ) ,
\end{gathered}
\end{equation}
where $ \psi_\Lambda $ is given by \eqref{eq:psiLa} and 
$$
|r(\alpha,\beta)|\leq Ce^{-\langle \alpha -\beta\rangle/Ch}.
$$
\end{lemm}
\begin{proof}
Formally,
\[  K_P ( \alpha, \beta ) = \frac{ |c|^2 } { h^{ \frac n 2 } } \frac 1 { ( 2 \pi h )^n } \int_{\RR^{2n} }\!\int_{\RR^{n}}   e^{ \frac i h ( \varphi ( \alpha, y ) +  \langle y - y' , \eta  \rangle - \varphi^* ( \beta, y'  ) ) } p ( y, \eta ) dy'd \eta dy , \]
and the critical points of the phase 
\[\begin{gathered} ( y,y',\eta ) \mapsto \varphi ( \alpha, y ) +  \langle y - y' , \eta  \rangle - \varphi^* ( \beta, y'  ) ,\\
\varphi ( \alpha, y ) = \langle \alpha_x - y , \alpha_\xi \rangle +
\tfrac i 2 ( \alpha_x -y)^2 , \ \ \ \varphi^* ( \beta , y ) = 
\overline{ \varphi ( \bar \beta, y ) } , \ \ \alpha, \beta \in \Lambda, \end{gathered} \]
are 
\[  \begin{gathered} y = y' = y_c = \tfrac 1 2 (  \alpha_x + \beta_x )
+ \tfrac i 2 ( \beta_\xi - \alpha_\xi ) , \ \ 
\eta = \eta_c = \tfrac 12 ( \alpha_\xi + \beta_\xi ) + \tfrac i 2 ( \alpha_x - \beta_x ) . \end{gathered} \]
The critical value of the phase is given by $ \psi_\Lambda $ in 
\eqref{eq:psiLa}. This gives a formal argument for \eqref{eq:TPSLa}.

To justify this, 
we first shift contours by 
$$
\eta\mapsto \eta+i\e\frac{y-y'}{\langle y-y'\rangle},
$$
which changes the phase to
\begin{align*}
&\langle \alpha_x-y,\alpha_{\xi}\rangle +\frac{i}{2}(\alpha_x-y)^2+\langle y-y',\eta\rangle +\langle y'-\beta_x,\beta_\xi\rangle +\frac{i}{2}(\beta_x-y')^2+i\e\frac{(y-y')^2}{\langle y-y'\rangle}
\end{align*}
Next, we shift contours in $ y$ and $ y'$:
$$
y\mapsto y+i\e \frac{\eta-\alpha_\xi}{\langle \eta-\alpha_\xi\rangle},\qquad y'\mapsto y'+i\e\frac{\beta_\xi-\eta}{\langle \beta_\xi-\eta\rangle }.
$$
This results in the phase
\begin{gather*}
\langle \alpha_x-y,\alpha_{\xi}\rangle +\frac{i}{2}(\alpha_x-y)^2+\langle y-y',\eta\rangle +\langle y'-\beta_x,\beta_\xi\rangle +\frac{i}{2}(\beta_x-y')^2\\
 + \, i\e\frac{(\eta-\alpha_\xi)^2}{\langle \eta-\alpha_\xi\rangle}+i\e\frac{(\eta-\beta_\xi)^2}{\langle \eta-\beta_\xi\rangle} +i\e\frac{(y-y')^2}{\langle y-y'\rangle}-\e\Big[\big\langle \alpha_x-y,\frac{\alpha_\xi-\eta}{\langle \alpha_\xi-\eta\rangle}\big\rangle+\big\langle \beta_x-y',\frac{\eta-\beta_\xi}{\langle \beta_\xi-\eta\rangle}\big\rangle\Big]\\
+ \, O\Big(\e^2\Big[\frac{|y-y|
^2}{\langle y-y'\rangle }+\frac{|\eta-\alpha_\xi|^2}{\langle \eta-\alpha_\xi\rangle^2}+\frac{|\beta_\xi-\eta|^2}{\langle \beta_\xi-\eta\rangle^2}\Big]\Big).
\end{gather*}
Therefore, choosing $\e>0$ small enough (not depending on $G$), we observe that the imaginary part of the phase satisfies
\begin{align*}
\Im \Phi&\geq \Im\left[\langle \alpha_x-y,\alpha_{\xi}\rangle +\frac{i}{2}(\alpha_x-y)^2+\langle y-y',\eta\rangle +\langle y'-\beta_x,\beta_\xi\rangle +\frac{i}{2}(\beta_x-y')^2\right]\\
&+\Im\left[i\e\frac{(\eta-\alpha_\xi)^2}{\langle \eta-\alpha_\xi\rangle}+i\e\frac{(\eta-\beta_\xi)^2}{\langle \eta-\beta_\xi\rangle} +i\e\frac{(y-y')^2}{\langle y-y'\rangle}\right]-\frac{M}{4}\e\left[ |\alpha_x-y|^2+|\beta_x-y'|^2\right]\\
&-\frac{\e}{M}\frac{|\alpha_\xi-\eta|^2}{|\langle \alpha_\xi-\eta\rangle|^2}+\frac{|\eta-\beta_\xi|^2}{|\langle \beta_\xi-\eta\rangle|^2}]-C\e^2\frac{|y-y|^2
}{\langle y-y'\rangle }\\
&\geq c|\alpha_x-y|^2+c|\beta_x-y'|^2+c\e|\eta-\alpha_\xi|+c\e |\eta-\beta_\xi|+c\e|y-y'|-C\|G\|_{C^1}.
\end{align*}
In the last line we have used that $G$ is compactly supported to see that 
$$
|\Im \left( \langle \alpha_x,\alpha_\xi\rangle-\langle \beta_x,\beta_\xi\rangle \right) |\leq C\|G\|_{C^1}.
$$

Now, suppose that $|\alpha-\beta|>\delta$.  Then, 
$$|\alpha_x-y|+|\beta_x-y'|+|y-y'|+|\alpha_\xi-\eta|+|\beta_\xi-\eta|>\delta$$
 and, choosing $\|G\|_{C^1}$ small enough depending on $\delta$, the integral is controlled by $Ce^{-\langle \alpha-\beta\rangle /Ch}$.

In particular, we have, modulo an acceptable error,
\[  K_P ( \alpha, \beta ) = \frac{ |c|^2 } { h^{ \frac n 2 } } \frac 1 { ( 2 \pi h )^n } \int_{\RR^{2n} }\!\int_{\RR^{n}}  e^{ \frac i h ( \varphi ( \alpha, y ) +  \langle y - y' , \eta  \rangle - \varphi^* ( \beta, y'  ) ) } p ( y, \eta )\chi(\delta^{-1}|\alpha-\beta|) \,dy'd \eta dy \]
where $\chi\in C_c^\infty(-2,2)$, $\chi\equiv 1$ on $[-1,1]$.

Since we are now working in a small neighborhood of the diagonal, the contour shift,
\begin{gather*}
y\mapsto y+y_c(\alpha,\beta)\qquad y'\mapsto y'+y_c(\alpha,\beta),\qquad
\eta\mapsto \eta+\eta_c(\alpha,\beta)
\end{gather*}
is justified. The phase after this contour shift is given by
$$
\frac{i}{4}[(\alpha_\xi-\beta_\xi)^2+(\alpha_x-\beta_x)^2]+\frac{1}{2}(\alpha_x-\beta_x,\alpha_\xi+\beta_\xi)+\frac{i}{2}y^2+\frac{i}{2}y'^2+\langle y-y',\eta\rangle.
$$
The stationary point of the phase is now at $y=y'=\eta=0$ and the imaginary part of the phase is always larger than at the critical point. Therefore, we may apply the method of steepest decent 
to obtain the expansion in \eqref{eq:TPSLa}.
\end{proof}

We now proceed as in the proof of Theorem \ref{t:2} to obtain
\begin{theo}
\label{t:3}
Suppose that $ P$ is given by \eqref{eq:pseudo1} where the symbol $ p $
enjoys a holomorphic extension satisfying
\begin{equation*} 
| p ( z, \zeta ) | \leq M , \ \ | \Im z |\leq a,  \ \ | \Im \zeta | \leq b .
\end{equation*}
For $ G \in C^\infty_{\rm{c}} ( T^* \RR^n ) $ with 
$ \| B \|_{ C^2 } $ sufficiently small we define 
\[ \Lambda = \Lambda_G := \{ ( x + i G_\xi ( x, \xi ) , \xi - 
i G_x ( x, \xi ) ) : ( x, \xi ) \in T^* \RR^n ) , \ \ 
L^2_\Lambda := L^2 ( \Lambda , e^{ - 2 H ( \alpha )/h} d \alpha ), \]
where $ H $ is given by \eqref{eq:formH}.  Let $ T_\Lambda u := 
T u |_\Lambda $ (see \eqref{eq:TLa}). 

Then, for $ u , v \in L^2 ( \RR^n ) $, 
\begin{equation}
\label{eq:t2}
\langle T P u , T v \rangle_{ L^2_\Lambda } 
= \langle M_{P_\Lambda} T u , T v \rangle_{ L^2_\Lambda } + 
\langle R_\Lambda T u , T v \rangle_{ L^2_\Lambda } , \end{equation}
where $R_\Lambda= \mathcal O ( h^\infty )_{ L^2_\Lambda \to L^2_\Lambda } $ and 
\[ \begin{gathered} 
 P_\Lambda ( z, \zeta , h ) = p |_\Lambda ( z, \zeta) + h p_\Lambda ^1 ( z, \zeta ) + \cdots, \ \ \ ( z, \zeta ) \in \Lambda . 
\end{gathered} \]
\end{theo}


\section{Weights vs. deformations}
\label{s:wvsd}

To show that the approaches of \S \ref{s:we} and \S \ref{s:projdefo} 
are the same, 
we want 
 to find $ \varphi = \varphi ( x, \xi ) \in \CIc ( \RR^{2n} )$ such that 
\begin{equation}
\label{eq:TSbd}   T S_\Lambda = \mathcal O ( 1 ) =  L^2_\Lambda \to L^2_\varphi , \ \ 
T_\Lambda S = \mathcal O ( 1 ) : L^2_\varphi \to L^2_\Lambda . 
\end{equation}
Let $ \varphi_G $ be the phase in $ T_\Lambda $ and 
$ \widetilde \varphi_G $ be the phase in $ S_\Lambda $. We need 
\begin{equation}
\label{eq:condph}
\begin{gathered}
\varphi ( x, \xi ) = \varphi_{\rm{max}} ( x , \xi ) = \varphi_{\rm{min}} ( x, \xi) , \\
\varphi_{\rm{max}}  ( x, \xi ) := \max_{ (x',\xi') \in \RR^{2n} } \left(
- \Im {\rm{c.v}}_y ( \varphi_0 ( x, \xi, y ) + \widetilde \varphi_G 
( x', \xi', y) )  {+} H ( x', \xi' ) \right)  \\
 \varphi_{\rm{min}} ( x, \xi ) :=  \min_{ (x',\xi') \in \RR^{2n} } 
 \left (  {-}H ( x', \xi' ) + \Im {\rm{c.v}}_y ( \widetilde \varphi_0 ( x, \xi, y ) +  \varphi_G  ( x', \xi', y) ) \right) .
 \end{gathered}
 \end{equation}
We start by noting that 
\[ \varphi_G ( x, \xi , y ) = \Phi ( z, \zeta, y ) |_{ ( z, \zeta ) \in \Lambda_G } , \ \ 
\Phi ( z, \zeta, y ) = ( z - y ) \zeta + \tfrac i 2 ( z - y )^2 , \]
and that 
\[  \widetilde \varphi ( x, \xi, y )  = - \bar \Phi ( z, \zeta, y ) |_{ ( z, \zeta ) \in \Lambda_G } , \ \ \bar \Phi ( z, \zeta ) = 
\overline{ ( \bar z , \bar \zeta ) }. \]
The critical value of $ y \mapsto \Phi ( z , \zeta, y ) - \bar \Phi ( x, \xi, y ) $
is given by 
\[  y_{\rm{c}} = y_{\rm{c}} ( x, \xi, z , \zeta ) = \tfrac 12 ( x + z + i ( \xi - \zeta )) ,\]
while the critical value of
$ y \mapsto \Phi ( x, \xi, y ) - \bar \Phi ( x, z, \zeta ) $
is given by 
\[ \bar  y_{\rm{c}} = \bar y_{\rm{c}} ( x, \xi, z , \zeta ) = \tfrac 12 ( x + z + i ( \zeta - \xi )) . \]
To find the maximum in \eqref{eq:condph} we { first note that   with $ z = x' + i G_{\xi'} $ and $ \zeta = \xi' - i G_{x'} $, 
\[ \begin{split} \Im \bar \Phi ( x', \xi', y_c ) & = 
\tfrac12 \Im \left( i ( \xi - \zeta ) \xi' - i ( x'- \tfrac 12 ( 
x + z + i ( \zeta - \xi ) ) )^2 \right) \\
& = - \tfrac1 4 ( 3  (\xi')^2 + (x')^2 ) + \mathcal O ( \langle \xi' \rangle + 
\langle x' \rangle ) \to - \infty , \ \  ( x',\xi') \to \infty. 
\end{split} 
\]
Hence, 
$$
- \Im  [\Phi ( x, \xi, y_c ) - 
\bar \Phi ( x' , \xi', y_c )]+H(x',\xi')\to -\infty,\qquad (x',\xi')\to \infty
$$}
We then calculate (again with 
$ z = x' + i G_{\xi'} $ and $ \zeta = \xi' - i G_{x'} $)
\[  \begin{split}
& d_{x',\xi'} ( - \Im {\rm{c.v.}}_y ( \Phi ( x, \xi, y ) - 
\bar \Phi ( z , \zeta, y ) )  = 
 \Im \partial_{z, \zeta } \bar \Phi ( z, \zeta,y   ) 
 |_{ y = \bar y _{\rm{c}}} 
   \end{split} .\]
Since $ d_{x',\xi' } H = - \Im \zeta dz |_{ \Lambda } $ this means that
the critical $ z , \zeta $ are given by solving
\[ \Im ( \partial_{z,\zeta} \bar \Phi ( z, \zeta, y )|_{ y = \bar y_{\rm{c}}}
 {-} \zeta dz ) |_{ \Lambda } = 0 .\]
For the minimum we similarly obtain
\[ \Im ( \partial_{z, \zeta}  \Phi ( z, \zeta , y ) |_{ y =  y_{\rm{c}}}
  {+} \zeta dz ) |_{\Lambda } = 0 .\]
This shows that the critical  {points} 
\[  (z_{\rm{c}} , \zeta_{\rm{c}} ) = (z_{\rm{c}} ( x, \xi) , \zeta_{\rm{c}} ( x, \xi ) ) ,\]
are the same for the maximum and minimum at \eqref{eq:condph}.
The maxima and minima are non-degenerate as that is the case 
when $ G = 0 $ and hence holds for $ \| G \|_{ C^2 } $ sufficiently small.

We now need to show that the critical values $ \varphi_{\max} $
and $\varphi_{\min} $ are also equal. For that we compute the differentials:
\[ \begin{split} d_{x,\xi} \varphi_{\rm{max}}  ( x , \xi) & = 
-\Im d_{x,\xi} \Phi ( x, \xi, y ) |_{ y= \bar y_{\rm{c}} ( 
x, \xi, z_{\rm{c}}, \zeta_{\rm{c}} )} 
 =  \Im d_{x, \xi} \bar \Phi ( x, \xi, y ) |_{ y = y_{\rm{c}}  ( 
x, \xi, z_{\rm{c}}, \zeta_{\rm{c}} )} \\
& = d_{x,\xi} \varphi_{\rm{min}} ( x, \xi ) . 
\end{split}\]
Hence $\varphi_{\rm{max}} $ and $ \varphi_{\rm{min}} $ differ by a constant. Since $ G $ and $ H $ vanish outside a compact set
the critical values are both $ 0 $ when $ H = G = 0 $, 
we conclude the the constant is equal to $ 0 $. This gives us

\begin{theo}
\label{th:G2ph}
There exist
$ \epsilon_0 $, $ C_0 $ such that if $ G \in \CIc ( \RR^{2n} ) $ and $ \| G \|_{C^2} < \epsilon_0 $ then
\begin{gather*}  \| T_\Lambda  v \|_{ L^2_\Lambda } /C_0 \leq \| T v \|_{ 
L^2_\varphi } \leq 
C_0 \| T_\Lambda v \|_{ L^2_\Lambda }, \ \ v \in L^2 ( \RR^n ) , \\
L^2_\Lambda := L^2 ( \Lambda, e^{ -2H( x, \xi)/h} dxd\xi) , \ \
L^2_\varphi  := L^2 ( T^* \RR^n , e^{-\varphi( x, \xi )/h} dx d\xi ), 
\end{gather*}
where $ \Lambda $, $ T_\Lambda $ are given in \eqref{eq:LaG},\eqref{eq:TLa}, $ H $ is defined by \eqref{eq:defH}, and 
$ \varphi $ is given (implicitely) by \eqref{eq:condph}.
\end{theo}
\begin{proof}
We have shown that for $ \varphi $ given by \eqref{eq:condph} we have 
\eqref{eq:TSbd}. Hence,
\[  \| T_\Lambda v \|_{ L^2_\Lambda } = 
\| T_\Lambda S T v \|_{L^2_\Lambda } \leq \| T_\Lambda S \|_{ L^2_\varphi
\to L^2_\Lambda } \| T v \|_{L^2_\varphi } \leq C_0 \| T v \|_{L^2_\varphi }, \]
with the other bound derived similarly.
\end{proof}


\end{document}